\documentclass{article}

\usepackage{amsmath,amssymb,makeidx,amsfonts,amsthm,latexsym,enumerate}
\usepackage{cite}

\theoremstyle{plain}

\newtheorem{lemma}{Lemma}[section]
\newtheorem{theorem}{Theorem}[section]
\newtheorem{corollary}{Corollary}[section]
\newtheorem{proposition}{Proposition}[section]
\newtheorem{definition}{Definition}[section]

\newtheorem*{theorem*}{Theorem}

\newtheorem{thmA}{Theorem}

\newtheorem*{summary of results}{Summary of results}

\theoremstyle{definition}

\newtheorem{remark}{Remark}[section]

\begin{document}

\author{Karl-Olof Lindahl\\
School of Mathematics and Systems Engineering\\
V\"{a}xj\"{o} University, 351 95, V\"{a}xj\"{o}, Sweden\\
\texttt{Karl-Olof.Lindahl@vxu.se}}

\title{Estimates of linearization discs in $p$-adic dynamics
with application to ergodicity\footnote{Revised version of preprint 04098 MSI, 2004,
V\"{a}xj\"{o} University, Sweden and part of the thesis \cite{Lindahl:2007}}}

\date{\empty}
\maketitle


\begin{abstract}
We give lower  bounds for the size of linearization discs for
power series over $\mathbb{C}_p$. For quadratic maps, and certain
power series containing a `sufficiently large' quadratic term, we
find the exact linearization disc. 
For finite extensions of $\mathbb{Q}_p$,
we give a sufficient condition on the multiplier under which the
corresponding linearization disc is maximal (i.e.\@ its radius coincides with 
that of the maximal disc in $\mathbb{C}_p$ on which $f$ is one-to-one).
In particular, in unramified extensions of $\mathbb{Q}_p$, the
linearization disc is maximal if the multiplier map has a maximal cycle on the unit sphere.      
Estimates of linearization discs in the remaining types of non-Archimedean fields of dimension one were obtained in 
\cite{Lindahl:2004,Lindahl:2009,Lindahl:2009eq}.

Moreover, it is shown that, for any
complete non-Archimedean field, transitivity is preserved under
analytic conjugation. Using results by Oxtoby
\cite{Oxtoby:1952}, we prove that transitivity,
and hence minimality, is equivalent the  unique ergodicity on
compact subsets of a linearization disc.   In particular, a power series $f$ over $\mathbb{Q}_p$
is minimal, hence uniquely ergodic, on all spheres inside a linearization disc
about a fixed point if and only if the multiplier is maximal. We also note that in finite extensions of $\mathbb{Q}_p$, as well as in any other non-Archimedean field $K$ that is not isomorphic to $\mathbb{Q}_p$ for some prime $p$, a power series cannot be ergodic on an entire sphere, that is contained in a linearization disc, and centered about the corresponding fixed point. 
\end{abstract}

\vspace{1.5ex}\noindent {\bf Mathematics Subject Classification
(2000):} 32P05, 32H50, 37F50, 37B05, 37A50

\vspace{1.5ex}\noindent {\bf Key words:} dynamical system,
conjugation, linearization, $p$-adic numbers, non-Archimedean field

\section{Introduction}

In this paper we study iteration of power series $f$ defined over
$\mathbb{C}_p$, the completion of the algebraic closure of the
$p$-adic numbers $\mathbb{Q}_p$. As in complex dynamics (i.e.\@
iteration of complex-valued analytic functions,
see e.g.\@ \cite{Carleson/Gamelin:1991,Milnor:2000,Beardon:1991}), the main features of the
dynamics under $f\in \mathbb{C}_p[[x]]$ is determined by the
character of the periodic points of $f$, i.e.\@ the modulus of the
multiplier at the periodic points. A periodic fixed point point
$x_0$ may be either \emph{attracting}, \emph{indifferent} or
\emph{repelling} depending on whether the multiplier
$\lambda=f'(x_0)$ is inside, on or outside the unit sphere. In
this paper we consider non-resonant (i.e.\@ $\lambda $ not a root of
unity) indifferent fixed points.

A power series over a complete valued field of the form
\[
f(x)=\lambda (x-x_0) + \text{(higher order terms)}
\]
is said to be linearizable at the fixed point $x_0$ if there
exists a
convergent power series solution $g$ to the following form of the
Schr\"{o}der functional equation (SFE)
\begin{equation}\label{schroder functional equation}
g\circ f(x)=\lambda g(x), \quad \lambda =f'(x_0),
\end{equation}
which conjugates $f$ to its linear part in some neighborhood of
$x_0$.
By the non-Archimedean Siegel theorem of Herman and Yoccoz
\cite{Herman/Yoccoz:1981},
as in the complex field case \cite{Siegel:1942}, the condition
\begin{equation}\label{Siegel condition}
|1- \lambda^n|\geq Cn^{-\beta} \quad\text{for some real numbers
$C,\beta >0$},
\end{equation}
on $\lambda$ is sufficient for convergence also in the non-Archimedean field
case. Their theorem applies to the multi-dimensional case.
In dimension one, the condition (\ref{Siegel condition}) is always satisfied for non-resonant multipliers in fields of characteristic zero, i.e. the $p$-adic case studied in this paper, and the equal characteristic case of studied in \cite{Lindahl:2009eq}.
This is not always true in fields of prime characteristic
as shown in \cite{Lindahl:2004,Lindahl:2009}.

As shown by Herman and Yoccoz, in the two-dimensional $p$-adic case there 
also exist examples where the Siegel condition is not satisfied and the corresponding conjugacy diverges. The multi-dimensional $p$-adic case has been taken further by Viegue in his thesis \cite{Viegue:2007}. In this paper we only consider the one-dimensional non-resonant $p$-adic case so the conjugacy always converges.

The conjugacy function $g$ is unique if we specify the image and
derivative at $x_0$. It is custom to assume that $g(x_0)=0$ and
$g'(x_0)=1$. By the local invertibility theorem, 
$g$ has a local inverse $g^{-1}$ at $x_0$.
We will refer to the \emph{(indifferent) linearization disc} of $f$ about $x_0$,
denoted by $\Delta _f(x_0)$, as the largest disc $U\subset
\mathbb{C}_p$, with $x_0\in U$, such that (\ref{schroder
functional equation}) holds for all $x\in U$, and $g$ converges
and is one-to-one on $U$. The possibly larger disc, on which the
the semi-conjugacy (\ref{schroder functional equation}) holds, will be referred to as 
the \emph{semi-disc}.

Note that, by definition, $f$ must be one-to-one on the linearization disc 
$\Delta _f(x_0)$. Moreover, the full conjugacy
\begin{equation}\label{full conjugacy}
g\circ f\circ g^{-1}(x)=\lambda x,
\end{equation}
is valid for all $x\in g(\Delta _f(x_0))$.
%
%
Let $f^{\circ n}$ denote the $n$-fold composition of $f$ with
itself. On $g(\Delta _f(x_0))$ we have $g\circ f^{\circ n}\circ
g^{-1}(x)=\lambda ^n x$. Hence, there is a one-to-one
correspondence between orbits under
$f$ and the multiplier map
$T_{\lambda}: x\mapsto \lambda x$, on $\Delta_f(x_0)$ and
$g(\Delta_f(x_0))$, respectively. In particular, since $\lambda $
is not a root of unity, $f$ can have no periodic points on the
linearization disc, except the fixed point $x_0$. However,
the semi-disc may contain other periodic points as well, as manifest in 
the papers \cite{Arrowsmith/Vivaldi:1994,Pettigrew/Roberts/Vivaldi:2001}.  In fact,  the semi-disc is contained in the \emph{quasi-periodicity domain} of $f$, defined as the interior of the set of points on the projective line 
$\mathbb{P}(\mathbb{C}_p)=\mathbb{C}_p\cup \{\infty\}$ that are recurrent by $f$. In the case that
$f$ is a rational function, Rivera-Letelier
 \cite{Rivera-Letelier:2000} gave several characterizations of the quasi-periodicity domain of $f$
 and described its local and global dynamics.  In particular, he proved that analytic components 
of the domain of quasi-periodicity, which are $p$-adic analogues of Siegel discs and Herman rings in complex dynamics, are open affinoids (that is, they have simple geometry), and contains infinitely many indifferent periodic points. 

Our aim in this paper is three-fold. First, we obtain lower
(sometimes optimal) bounds for the size of linearization discs for
$f\in\mathbb{C}_p[[x]]$. These estimates extend results on
quadratic polynomials over $\mathbb{Q}_p$ by Ben-Menahem \cite{Ben-Menahem:1988},
and Thiran, Verstegen, and Weyers \cite{Thiran/EtAL:1989}, and for certain polynomials
with maximal multipliers over the $p$-adic integers $\mathbb{Z}_p$
by Pettigrew, Roberts and Vivaldi \cite{Pettigrew/Roberts/Vivaldi:2001}, as well as results on small
divisors in $\mathbb{C}_p$ by Khrennikov \cite{Khrennikov:2001a}. 

Second, we
prove that transitivity (the existence of a dense orbit) is
preserved under analytic conjugation into linearization discs over an
arbitrary complete non-Archimedean field. Using results by Oxtoby
\cite{Oxtoby:1952}, the transitivity of $f$ on
compact subsets of a linearization disc is proven to be equivalent to the
ergodicity and unique ergodicity of $f$.

Third, when the dynamics is defined over the $p$-adic numbers
$\mathbb{Q}_p$, we give necessary and sufficient conditions on the
multiplier, that $f$ is transitive, hence uniquely ergodic,  on
spheres inside the linearization disc. These results generalize results
obtained by Bryk and Silva \cite{Bryk/Silva:2003}, and by Gundlach, Khrennikov, and Lindahl 
\cite{Gundlach/Khrennikov/Lindahl:2001:a}, for
monomials $f:x\mapsto cx^n$, and for 1-Lipschitz power series by Anashin \cite{Anashin:2006}.
On the other hand, we also prove
that transitivity is not possible on a whole sphere in any proper 
extension of $\mathbb{Q}_p$. Results on the transitivity and ergodic breakdown of the $p$-adic
multiplier map $x\mapsto \lambda x$ were obtained by Oselies and Zieschang \cite{Oselies/Zieschang:1975}, and by Coelho and Parry \cite{Coelho/Parry:2001}.

A classification of measure-preserving transformations of compact-open subsets of non-Archimedean local fields were obtained recently by Kingsbery, Levin, Preygel and Silva
\cite{KingsberyLevinPreygelSilva:2009}. They show that if a $C\sp 1$ transformation $T$ is
measure-preserving when restricted to a compact-open set $X$ then $X$ can be written
as a disjoint union of invariant compact-open sets such that $T$ restricted to each such set is either a local isometry or topologically and measurably conjugate to an ergodic  Markov transformation. Concerning polynomials, the question of ergodicity is also answered,
except in the case where the polynimial is $1$-Lipschitz, as in the present paper  (the power series $f$ is certainly $1$-Lipschitz on the entire linearization disc). Non-1-Lipschitz functions were also studied in \cite{AnashinKhrennikov:2009}.

Let us also mention some related works on measure preserving transformations on the Berkovich space, which is a much larger space than the $p$-adics.
The Berkovich space provides a bridge between non-archimedean and complex dynamics.  The  works \cite{RumelyBaker:2004,FavreRivera-Letelier:2004,FavreRivera-Letelier:2006}, 
construct a natural invariant measure for a wide class of rational functions, similar to existing constructions in complex dynamics. 


Further results on the properties of the dynamics on $p$-adic
linearization discs are provided in
\cite{Arrowsmith/Vivaldi:1994,Pettigrew/Roberts/Vivaldi:2001}. 
Estimates for linearization discs in prime characteristic were obtained in
\cite{Lindahl:2004,Lindahl:2009}, and for fields of charactersitic zero in the equal charactersitic case \cite{Lindahl:2009eq}. See \cite{Lindahl:2004}, for
further comments on the non-Archimedean problem of linearization
and its relation to the complex field case.

The construction of cojugacies in $p$-adic
dynamics is related to standard and well-established techniques of
local arithmetic geometry, see e.g.\@ Lubin \cite{Lubin:1994} and  the
construction of local canonical heights by Call and Silverman \cite{CallSilverman:1993},  and Hsia
\cite{Hsia:1996}. For indifferent,
non-resonant, fixed points the conjugacy function is related to
the `logarithm' of the theory of one-parameter formal Lie-groups
defined over the $p$-adics
\cite{Arrowsmith/Vivaldi:1994,Lubin:1994}. As in
\cite{Lubin:1994}, the Lie-logarithm is constructed as the limit
\begin{equation}\label{Lie-logarithm}
\lim_{n\to\infty}\frac{f^{\circ p^n}-id_x}{p^n},
\end{equation}
and is, up to a constant, the quotient between the conjugacy
function $g$ and its derivative $g'$.
The Lie-logarithm contains useful information about the dynamics
of $f$. In particular, its roots are periodic points of $f$
\cite{Lubin:1994}. See Li
\cite{Li:1996a,Li:1996c} for various results on
this matter, including the counting of periodic points of $p$-adic
power series. Rivera-Letelier \cite{Rivera-Letelier:2000} 
proved if $f$ is a rational function, then the Lie-logarithm converges
uniformly on the entire domain of quasi-periodicity.

For some additional references on non-Archimedean dynamics and its
relationship, similarities, and differences with respect to the
Archimedean theory of complex dynamics, see e.g.\@
\cite{Arrowsmith/Vivaldi:1993,Benedetto:2000c,
Benedetto:2001a,Benedetto:2002, Benedetto:2003a,
Bezivin:2004a,Bezivin:2004b,Hsia:1996,Hsia:2000,Khrennikov:2001a,
Rivera-Letelier:2003,DeSmedtKhrennikov:1997,
AnashinKhrennikov:2009,Khrennikov/Nilsson:2004,
Khrennikov:2003ryssbok,DragovichKhrennikovMihajlovic:2007,
KhrennikovMukhamedovMendes:2007,KhrennikovSvensson:2007,
Khrennikov:2003nauk, Svensson:2005,NilssonNyqvist:2004}.
%
%
Applications of $p$-adic numbers have  been proposed in coding
theory \cite{CalderbankSloane:1995},
round off errors \cite{Bosio/Vivaldi:2000}, random number
generation \cite{WoodcockSmart}, and in biochemistry and physics
\cite{AvetisovBikulovKozyrevOsipov:2002,BaakeMoodySchlottmann:1998,Khrennikov:1997,Khrennikov:2004a
,RammalToulouseVirasoro:1986}.

\section{Summary of results}

Our most general result on the size of a linearization disc in $\mathbb{C}_p$ can be stated in the
following way (see also Theorem \ref{theorem general estimate} and
Lemma \ref{lemma upper bound Siegel and isometry}).

\begin{thmA}[Estimate of linearization discs in $\mathbb{C}_p$] \label{theorem linearization disc cp general}
Let $f\in\mathbb{C}_p[[x]]$ have an indifferent
fixed point $x_0$, with multiplier $\lambda =f'(x_0)$, not a root
of unity. Suppose that $f$ has the following
expansion about $x_0$
\begin{equation}\label{definition power series about x0}
f(x)=x_0+\lambda (x-x_0)+\sum_{i\geq 2}a_i(x-x_0)^i, \quad
\textrm{with } a=\sup_{i\geq 2}|a_i|^{1/(i-1)}.
\end{equation}
Then, the linearization disc $\Delta_f(x_0)$, satisfies
$D_{\sigma(\lambda,a)}(x_0)\subseteq \Delta_f(x_0)\subseteq
\overline{D}_{1/a}(x_0)$,
where $\sigma(\lambda,a)$ is defined by (\ref{definition
sigma}). Moreover, if the conjugacy function $g$ converges on the
closed disc $\overline{D}_{\sigma(\lambda,a)}(x_0)$, then
$\Delta_f(x_0)\supseteq \overline{D}_{\sigma(\lambda,a)}(0)$. In
particular, $f$ can have no periodic points in the punctured open
disc $D_{\sigma(\lambda,a)}(x_0)\setminus \{x_0\}$.
\end{thmA}

The proof is based on estimates of the coefficients of the
conjugacy function $g$. Applying a result of Benedetto
\cite{Benedetto:2003a} (Proposition \ref{proposition-discdegree}
below), on these estimates we find a lower bound for the region of
convergence of the inverse $g^{-1}$, and hence of the linearization disc.

Note that the estimate $\sigma=\sigma(\lambda,a)$ depends only on
$\lambda$ and the real number $a$. To find the exact size of the
linearization disc we do in general need more information about the
coefficients of $f$. However, for a large class of quadratic polynomials,
and certain
power series containing a `sufficiently large' quadratic term, we
prove that
\begin{equation}
\tau=|1-\lambda|^{-1/p}\sigma(\lambda,a)
\end{equation}
is the exact radius of the linearization disc. More precisely, our main result can
be stated in the following way  (see also Theorem
\ref{theorem quadratic polynomials}).
\begin{thmA}[Linearization disc for quadratic maps]\label{thmA quadratic
maps}
Let $p$ be an odd prime. Let
\[f(x)=x_0+\lambda (x-x_0)+
a(x-x_0)^2\in\mathbb{C}_p[x-x_0],
\]
with $\lambda$ not a root of unity. Suppose that
$p^{-1}<|1-\lambda |<1$. Then, the linearization disc $\Delta_f(x_0)$ is
equal to the disc $D_{\tau(\lambda,a)}(x_0)$, where the radius
$\tau(\lambda,a)=|1-\lambda |^{-1/p}\sigma(\lambda,a)$.

\end{thmA}
This result is extended in Theorem \ref{theorem quadratic power
series} to power series containing a `sufficiently large'
quadratic term. We also give sufficient conditions, there being a
fixed point on the `boundary' of the linearization disc, i.e.\@ the sphere
$S_{\tau}(x_0)$ about $x_0$ of radius $\tau$. 

Note that $\tau(\lambda,a)<1/a$.
Hence, at least in this case, the linearization disc cannot contain the maximal disc $D_{1/a}(x_0)$
on which $f$ is one-to-one.

The relatively complicated expression for $\sigma$
stems from the
presence of $p^s$th roots of unity in the punctured disc
$D_1(1)\setminus \{1\}$, as described in Section \ref{section
geometry and roots of uni ty C p}. Some properties of $\sigma$ are
discussed in Section \ref{section asympt behav siegel rad}.
In particular, we prove the following result.
\begin{thmA}[Asymptotic behavior of $\sigma$]\label{theorem asymptotic behavior sigma}
Let $|\alpha -\lambda ^m|$ be fixed. Then, the estimate $\sigma$
of the radius of the linearization disc goes to $1/a$ as $m$ or $s$ goes to infinity.
If $s$ and $m$ are fixed, then $\sigma \to 0$ as $|\alpha -\lambda
^m|\to 0$.  
\end{thmA}


We now turn to the special case when the dynamics is restricted to $\mathbb{Q}_p$ and its finite extensions. In $\mathbb{Q}_p$, there are no $p^s$th roots
of unity in $D_1(1)\setminus \{1\}$, and $\sigma $ takes a simpler
form.

\begin{thmA}[linearization discs in $\mathbb{Q}_p$ for odd primes]\label{theoremA linearization disc in Qp}
Let $p$ be an odd prime, and let $f\in\mathbb{Q}_p[[x-x_0]]$ be of
the form (\ref{definition power series about x0}).
Let $\Delta _f(x_0,\mathbb{Q}_p)=\Delta _f(x_0)\cap \mathbb{Q}_p$ be the corresponding linearization disc
in $\mathbb{Q}_p$. Then, $\Delta_f(x_0,\mathbb{Q}_p)\supseteq
D_{\sigma_1}(x_0,\mathbb{Q}_p)$, where
\[
\sigma_1=a^{-1}p^{-\frac{1}{m(p-1)}}|1-\lambda ^m|^{\frac{1}{m}},
\]
and $m\geq 1$ is the smallest integer such that $|1-\lambda
^m|<1$.
Furthermore, if $|1-\lambda ^m|=p^{-1}$ and $m=p-1$, then
$\Delta_f(x_0,\mathbb{Q}_p)$  is either the open or closed disc of radius $1/a$ about $x_0$. In particular,  if
either $\max_{i\geq 2 }|a_i|^{1/(i-1)}$ is attained (as for polynomials) or $f$
diverges on $S_{1/a}(x_0,\mathbb{Q}_p)$, then $\Delta_f(x_0,\mathbb{Q}_p)=D_{1/a}(x_0,\mathbb{Q}_p)$.
\end{thmA}

Note that the condition $|1-\lambda ^m|=p^{-1}$ and $m=p-1$, imply that $\lambda $ has a maximal cycle modulo $p^2$ in the sense that it is a generator of the group of units $(\mathbb{Z}/p^2\mathbb{Z})^{*}$. In this case $\lambda $ is said to be \emph{maximal}.

\begin{thmA}[linearization discs in $\mathbb{Q}_2$]\label{theoremA linearization disc in Q2}

Let $f\in\mathbb{Q}_2[[x-x_0]]$ be of
the form (\ref{definition power series about x0}). 
Then, the following two statements hold: 
\begin{enumerate}

\item If $|1-\lambda |<1/2$, then the linearization disc $\Delta_f(x_0,\mathbb{Q}_2)$ contains the open disc of radius
         $\sigma_1= |1-\lambda |/2a $ about $x_0$.

\item If $|1-\lambda |=1/2$, then the linearization disc $\Delta_f(x_0,\mathbb{Q}_2)$ contains the open disc of radius   $\sigma_3=\sqrt{ |1+\lambda|}/a$ about $x_0$.

\end{enumerate}

\end{thmA}

\begin{thmA}[Maximal linearization discs in extensions of $\mathbb{Q}_p$]\label{theoremA maximal
 linearization disc in K }
Let $K$ be a finite extension of $\mathbb{Q}_p$ of degree $n$, with ramification index $e$, residue field  $k$ of  degree $[k:\mathbb{F}_p]=n/e$,
and uniformizer $\pi$. Let $f\in K[[x]]$ be a power series of the form (\ref{definition power series about x0})  and  $\alpha$ a root of unity 
such that 
there is no closer root 
of unity to $\lambda^{p^{n/e}-1}$ than $\alpha$.

Suppose that $\lambda$ has a maximal cycle
modulo $\pi^2 $ and 
\[
\log_p e  \leq (p^{n/e} -3)p/(p-1)- \nu\left ( \frac{\alpha - \lambda ^{p^{n/e}-1}}{1-\lambda^{p^{n/e}-1}} 
\right )
+\log_p(p-1),
\]
where $\nu$ is the valuation. Then, the linearization disc $\Delta_f(x_0,K)=\Delta_f(x_0)\cap K$ is maximal
in the sense that $\Delta_f(x_0,K)$ is either the open or closed disc of radius $1/a$. In particular,  if
either $\max_{i\geq 2 }|a_i|^{1/(i-1)}$ is attained (as for polynomials) or $f$
diverges on $S_{1/a}(x_0)$, then $\Delta_f(x_0,K)=D_{1/a}(x_0,K)$.
\end{thmA}

Note that if the ramification index $e$ is not divisible by $p-1$, then $\alpha =1$ so that
the $\nu$-term vanishes in this case. 
Also note that the linearization disc may be maximal even if $\lambda$ does not have a maximal cycle
modulo $\pi^2 $, see Theorem \ref{theorem maximal linearization disc in K }.

In the final section of this paper we note some facts concerning transitivity, minimality and ergodicity on linearization discs.  In particular, we show that transitivity is preserved under
analytic conjugation into a linearization disc.  More precisely.

\begin{thmA}[Transitivity and conjugation in non-Archimedean fields]\label{thmA transitivity preserved}
Let $K$ be a complete non-Archimedean field. Suppose that the power series $f(x)=x_0 +
\lambda (x-x_0) +O((x-x_0)^2)\in K[[x-x_0]]$ is
analytically conjugate to $T_{\lambda}$, on the linearization disc
$\Delta_f(x_0)$ in $K$, via a conjugacy function $g$, with
$g(x_0)=0$ and $|g'(x_0)|=1$. Suppose also that the subset $X\subseteq
\Delta_f(x_0)$ is invariant under $f$. Then, the following
statements hold:
\begin{enumerate}[1)]

 \item $f$ is transitive on $X$ if and only if $T_{\lambda}$ is transitive on $g(X)$.
 \item If $X$ is compact and $f$ is transitive on
$X$, then $f$ is minimal on $X$. Moreover, $f(X)=X$ and
$g(X)=T_{\lambda}(g(X))$.
\end{enumerate}
Moreover, if $X$ is compact, the following are equivalent

\end{thmA}

In fact, the minimality of $f$ is equivalent to its unique ergodicity.

\begin{thmA}[Unique ergodicity in non-Archimedean fields]\label{thmA minimlity ergodicity subset}
Let $K$, $f$, $\Delta
_f(x_0)$, and $g$ be as in Theorem \ref{thmA transitivity
preserved}.
Suppose that the subset $X\subset \Delta_f(x_0)$ is non-empty,
compact and invariant under $f$. The following statements are
equivalent:
\begin{enumerate}[1.]
      \item $T_{\lambda} : g(X)\to g(X) $ is minimal.
      \item $f: X\to X$ is minimal.
      \item $f: X\to X$ is uniquely ergodic.
      \item $f$ is ergodic for any $f$-invariant measure $\mu$ on the Borel sigma-algebra $\mathcal{B}(X)$ that is
      positive on non-empty open sets.
    \end{enumerate}

\end{thmA}
The unique invariant measure $\mu$ is the normalized
Haar measure $\mu$ for which the measure of a disc is equal to the
radius of the disc.

Note that the conjugacy function $g$ maps spheres in the linearization
disc into spheres about the origin. In $\mathbb{Q}_p$, the
multiplier map $T_{\lambda}:x\mapsto \lambda x $ is minimal on
each sphere $S$ about the origin if and only if $\lambda$ is a
generator of $(\mathbb{Z}/p^2\mathbb{Z})^{*}$. Moreover, if
$\lambda$ is a generator of $(\mathbb{Z}/p^2\mathbb{Z})^{*}$, then
as a consequence of Theorem \ref{theoremA linearization disc in Qp}, the
linearization disc $\Delta _f(x_0,\mathbb{Q}_p)$ includes the the open
disc $D_{1/a}(0)$.

\begin{thmA}[Ergodic spheres in $\mathbb{Q}_p$]\label{thmA ergodic discs in Qp}
Let $p$ be an odd prime, and let the series $f\in\mathbb{Q}_p[[x-x_0]]$ be of
the form (\ref{definition power series about x0}). Let $S\subset
\mathbb{Q}_p$ be a non-empty sphere of radius $r<1/a$ about $x_0$,
i.e.\@ $r$ is an integer power of $p$. Then, the following
statements are equivalent:
\begin{enumerate}[1.]
      \item $\lambda $ is a generator of
      $(\mathbb{Z}/p^2\mathbb{Z})^{*}$.
      \item $f:S\to S$ is minimal.
      \item $f: S\to S$ is uniquely ergodic.
      \item $f$ is ergodic for any $f$-invariant measure $\mu$ on the Borel sigma-algebra $\mathcal{B}(S)$ that is
      positive on non-empty open sets.
    \end{enumerate}

\end{thmA}
By Theorem \ref{theoremA linearization disc in Qp} the estimate of the
radius $1/a$ is maximal in the sense that there exist examples of
such $f$, which either diverges on the sphere $S_{1/a}(x_0)$ or
satisfy $f(x)=x_0$ for at least one $x\in S_{1/a}(x_0)$.  We have,
however, not been able to rule out the possibility that in some
cases we may allow $r=1/a$, see Lemma \ref{lemma  f one-to-one}.

Also note that if $\lambda\in S_1(0)$ is not 
a generator of $(\mathbb{Z}/p^2\mathbb{Z})^{*}$, then $T_{\lambda }$ and hence
$f(x)=\lambda x +O(x^2)$ may still be minimal on some
subset of a sphere. A complete classification of the ergodic
breakdown of $\mathbb{Q}_p$ with respect to $T_{\lambda}$ is given in 
\cite{Oselies/Zieschang:1975}.

We also note (lemma \ref{lemma transitivity multiplier}) that in a finite proper extension of $\mathbb{Q}_p$, a power series   
cannot be ergodic on an entire sphere, that is contained in a linearization disc, and centered
about the corresponding fixed point.
In fact, if $K$ is a non-Archimedean field, then ergodicity on a linearization sphere is only possible if $K$
is isomorphic to a field of $p$-adic numbers. For transitivity  to occur, $K$ must be locally compact. Therefore, $K$ is either a $p$-adic field or a field of prime characteristics. Let $K$ be a locally compact field of prime characteristc, with uniformizer $\pi$. If $x\in K$ and $x \equiv 1 \mod \pi$, then $x ^{p^n} \equiv 1\mod \pi ^{p^n}$. As a 
consequence, $T_{\lambda}$ cannot be transitive on a sphere in $K$, see Lemma \ref{lemma non transitivity multiplier char p}.

\begin{thmA}[Ergodic non-Archimedean linearization spheres] \label{theorem non-archimedean ergodic disc}
Let $K$ be a complete non-Archimedean field and let $f$ be holomorphic on a disc $U$ in $K$. Suppose that $f$ has a linearization disc $\Delta\subset U$ and $S\subset \Delta $ is a sphere about the corresponding fixed point $x_0\in K$. Then $f$: $S\to S$
is ergodic if and only if $K$ is isomorphic to $\mathbb{Q}_p$ and the multiplier is a generator of the group of units 
$(\mathbb{Z}/p^2\mathbb{Z})^*$. Furthermore, if $K=\mathbb{Q}_p$ and $\lambda $ is a generator of the group of units 
$(\mathbb{Z}/p^2\mathbb{Z})^*$, then the radius of $\Delta$ is $1/a$ (considered as a disc in $\mathbb{Q}_p$).

\end{thmA}


\section{Preliminaries}\label{section preliminaries}


Throughout this paper $K$ is a non-Archimedean field, complete with respect to a nontrivial absolute value $|\cdot |$. That is, $|\cdot |$ is a multiplicative function
from $K$ to the nonnegative real numbers with $|x|=0$
precisely when $x=0$, satisfying the  
following
strong or ultrametric triangle inequality:
\begin{equation}\label{sti}
|x+y| \leq  \max[|x|,|y|],\quad\text{for all $x,y\in K$},
\end{equation}
and nontrivial in the sense that it is not
identically $1$ on $K^*$, the set of all nonzero elements
in $K$. 
One useful consequence of ultrametricity is that for any $x,y\in
K$ with $|x|\neq |y|$, the inequality (\ref{sti}) becomes an
equality. In other words, if $x,y\in K$ with $|x|<|y|$, then
$|x+y|=|y|$.

In this context it is standard to denote by $\mathcal{O}$, the ring of integers of $K$, given by
$\mathcal{O}=\{x\in K : |x|\leq 1\}$, by $\mathcal{M}$ the unique maximal ideal of
$\mathcal{O}$, given by $\mathcal{M}=\{x\in K: |x|<1\}$, and by $k$ the corresponding \emph{residue field} 
\[
k=\mathcal{O}/\mathcal{M}.
\]
Note that if $K$ has positive characteristic $p$, then also char $k =p$;  but if char $K=0$, then
$k$ could have characteristic $0$ or $p$. 
Note also that if $x,y\in\mathcal{O}$ reduce to \emph{residue classes}
$\overline{x},\overline{y}\in k$, then $|x-y|$ is $1$ if $\overline{x}\neq\overline{y}$, and it is
strictly less than $1$ otherwise. 

In this paper we mainly consider the case when $K$ is either a $p$-adic field, i.e.\@ a finite
 extension of a  field of  $p$-adic numbers $\mathbb{Q}_p$, or a field of complex $p$-adic numbers
 $\mathbb{C}_p$.
 Recall that the $p$-adic numbers are constructed in the following way. For any prime $p$,
 there is a unique absolute value on $\mathbb{Q}$ such that $|p|=1/p$. The field $\mathbb{Q}_p$
 of $p$-adic rationals is defined to be the corresponding completion of $\mathbb{Q}$;
  $\mathbb{C}_p$ is then the completion of an algebraic closure of $\mathbb{Q}_p$.
  Let us also  remark
  that the residue field of $\mathbb{Q}_p$ is the field $\mathbb{F}_p$, of $p$ elements,
  whereas the the residue field of $\mathbb{C}_p$ is the algebraic closure of $\mathbb{F}_p$.
 
Given $K$ with absolute value $|\cdot|$ we define the
\emph{value group} as the image
\begin{equation}\label{def-value group}
|K^{*}|=\{|x|:x\in K^* \}.
\end{equation}
Note that, since $|\cdot |$ is multiplicative, $|K^*|$ is a multiplicative subgroup of the positive
real numbers. We will also consider the full image
$|K|=|K^*|\cup\{0\}$. The absolute value $|\cdot |$ is said to be
\emph{discrete} if the value group is cyclic, that is if there 
is a $\emph{uniformizer}$ $\pi\in K$ such that $|K^{*}|=\{|\pi |^n: n\in
\mathbb{Z}\}$. Note that if $K=\mathbb{Q}_p$, then $p$ is a uniformizer of $K$, and the 
value group consists of all integer powers of $p$. If $K=\mathbb{C}_p$, then 
$|K^{*}|$ consists of all rational powers of $p$. In
particular, the absolute value on $\mathbb{C}_p$ is not discrete.

Recall that $K$ is locally compact (w.r.t. $|\cdot |$) if and only if 
(i) $|\cdot |$ is discrete, and (ii) the residue field $k$ is finite.
If $K$ is a $p$-adic field, then $K$ is locally compact and each integer 
$x\in \mathcal{O}$ has a unique
representation
as a Taylor series in $\pi$ of the form
\begin{equation}\label{equation expansion}
x=\sum_{i=0}^{\infty}x_i\pi^i, \quad x_i\in \mathcal{R},
\end{equation}
where  $ \mathcal{R}$ is a  complete
system  of representatives of the residue field $k$.

Given a prime $p$, a $p$-adic number $x$ can be expressed in base
$p$ as
\[
x=\sum_{k=\nu}^{\infty}x_kp^k, \quad x_k\in\{ 0, ..., p-1 \},
\]
for some integer $\nu$ such that $x_{\nu}\neq 0$ and $x_k=0$ for
all $k<\nu$. The absolute value of $x$ is given by $|x|=p^{-\nu}$.
If $x$ is an integer, its $p$-adic expansion contains no negative
powers of $p$ and hence $|x|\leq 1$.

For future reference, let us note the following lemma.
\begin{lemma}\label{lemma order of p in factorial}
Given a rational number $x$, denote by $\lfloor x\rfloor$ the
integer part of $x$. Let $n\geq 1$ be an integer and let $S_n$ be
the sum of the coefficients in the $p$-adic expansion of $n$.
Then,
\begin{equation}\label{equation order of p in factorial}
\nu(n!)=\frac{n-S_n}{p-1}\leq \frac{n-1}{p-1},
\end{equation}
with equality if $n$ is a power of $p$. Consequently, for all
integers $a\geq 1$,
\begin{equation}\label{limit order of p in factorial}
\frac{\nu(\left\lfloor \frac{n}{a}\right \rfloor!)} {n} \to
\frac{1}{a(p-1)},
\end{equation}
as $n$ goes to infinity.
\end{lemma}
\noindent For a proof of (\ref{equation order of p in factorial}),
the reader can consult \cite[Lemma 25.5]{Schikhof:1984}.

Furthermore, to each finite extension $K$ of $\mathbb{Q}_p$ of degree
$n$, there is an associated \emph{residue class degree}  $f=[k:\mathbb{F}_p]$,
and a \emph{ramification index} $e$ such that 
\begin{equation}\label{def-ramification index}
|K^*|=\{p^{l/e}: l\in\mathbb{Z}\}.
\end{equation}
For example, by adjoining $\sqrt p$ to
$\mathbb{Q}_p$ we get a ramified extension with ramification index
$e=2$. The degree of the extension $n=[K:\mathbb{Q}_p]$, the residue class degree $f$, and the ramification index $e$
satisfy the relation
\[
n=e\cdot f.
\]
A finite extension of degree $n$ is called 
\emph{unramified}, if $e=1$ (or equivalently, $f=n$), and \emph{ramified},
if $e>1$ (or equivalently, $f<n$).  

For more information on $p$-adic numbers and their field
extensions the reader can consult
\cite{Gouvea:1997}.

\subsection{Non-Archimedean discs}\label{section non-Archimedean
discs}

Let $K$ be a complete non-Archimedean field. Given an element
$x\in K$ and real number $r>0$ we denote by $D_{r}(x)$ the open
disc of radius $r$ about $x$, by $\overline{D}_r(x)$ the closed
disc, and by $S_{r}(x)$ the sphere of radius $r$ about $x$.
To omit confusion, we sometimes write $D_r(x,K)$ rather than $D_{r}(x)$
to emphasize that the disc is considered as a disc in $K$.  

If $r\in|K^*|$ (that is if $r$ is actually the absolute value of some
nonzero element of $K$), we say that $D_{r}(x)$ and
$\overline{D}_r(x)$ are \emph{rational}. Note that $S_r(x)$ is
non-empty if and only if $\overline{D}_r(x)$ is rational. If
$r\notin |K^*|$, then we will call $D_{r}(x)=\overline{D}_r(x)$ an
\emph{irrational} disc. In particular, if $a\in K\subset
\mathbb{C}_p$ and $r=|a|^s$ for some rational number
$s\in\mathbb{Q}$, then $D_{r}(x)$ and $\overline{D}_r(x)$ are
rational considered as discs in the algebraic closure
$\mathbb{C}_p$. However, they may be irrational considered as
discs in $K$. Note that all discs are both open and closed as
topological sets, because of ultrametricity.
%
%
However, as we will see in Section \ref{section non-Archimedean
power series} below, power series distinguish between rational
open, rational closed, and irrational discs.





\subsection{Non-Archimedean power series}\label{section non-Archimedean power series}

Let $K$ be a complete non-Archimedean field with absolute value
$|\cdot |$.
Let $f$ be a power series over $K$ of the form
\begin{equation*}
f(x)=\sum_{i=0}^{\infty}a_i(x-\alpha )^i, \quad a_i\in K.
\end{equation*}
Then, $f$ converges on the open disc $D_{R_f}(\alpha )$ of radius
\begin{equation}\label{radius of convergence}
R_f = \frac{1}{\limsup |a_i| ^{1/i}},
\end{equation}
and diverges outside the closed disc $\overline{D}_{R_f}(\alpha )$
in $K$. The power series $f$ converges on the sphere
$S_{R_f}(\alpha )$ if and only if
\[
\lim_{i\to\infty}|a_i| R_f ^i=0.
\]
\begin{definition}
Let $U\subset K$ be a disc, let $\alpha\in U$ and let $f:U\to K$.
We say that $f$ is {\bf holomorphic} on $U$ if we can write
$f$ as a power series
\[
f(x)=\sum_{i=0}^{\infty}a_i(x-\alpha)^i\in K[[x-\alpha]]
\]
which converges for all $x\in U$.
\end{definition}
Holomorphicity is well-defined since,   
contrary to the complex field case, it does not matter which $\alpha\in U$
we choose in the definition of holomorphicity, see e.g.\@ \cite{Schikhof:1984}.

The basic mapping properties of non-Archimedean power series on discs
are given by the following generalization by Benedetto \cite{Benedetto:2003a},
of the  Weierstrass Preparation Theorem \cite{BoschGuntzerRemmert:1984,FresnelvanderPut:1981,
Koblitz:1984}.

\begin{proposition}[Lemma 2.2 \cite{Benedetto:2003a}]\label{proposition-discdegree}
Let $K$ be algebraically closed. Let
$f(x)=\sum_{i=0}^{\infty}a_i(x-\alpha)^i$ be a nonzero power
series over $K$ which converges on a rational closed disc
$U=\overline{D}_R(\alpha)$, and let $0<r\leq R$. Let
$V=\overline{D}_r(\alpha)$ and $V'=D_r(\alpha)$. Then
  \begin{eqnarray*}
    s &=& \max\{|a_i|r^i:i\geq 0\},\\
    d &=& \max\{i\geq 0:|a_i|r^i=s\},\quad and\\
    d'&=& \min\{i\geq 0:|a_i|r^i=s\}
  \end{eqnarray*}
are all attained and finite. Furthermore,
\begin{enumerate}[a.]
 \item $s\geq |f'(x_0)|\cdot r$.
 \item if $0\in f(V)$, then $f$ maps $V$ onto $\overline{D}_s(0)$
 exactly $d$-to-1 (counting multiplicity).
 \item if $0\in f(V')$, then $f$ maps $V'$ onto $D_s(0)$
 exactly $d'$-to-1 (counting multiplicity).
\end{enumerate}
\end{proposition}

We will consider the case $a_0=0$ in more detail. For our purpose,
it is then often more convenient to state Proposition
\ref{proposition-discdegree} in the following way.

\begin{proposition}\label{proposition one-to-one}
Let $K$ be algebraically closed and let
$h(x)=\sum_{i=1}^{\infty}c_i(x-\alpha )^i$ be a power series over
$K$.
\begin{enumerate}[1.]

 \item Suppose that $h$ converges on the rational closed disc
  $\overline{D}_R(\alpha)$. Let $0<r\leq R$ and suppose that
  \begin{equation}\label{ck inequality one-to-one}
   |c_i|r^i\leq |c_1|r\quad \text{ for all } i\geq 2 .
  \end{equation}
 Then, $h$ maps the open disc $D_{r}(\alpha )$ one-to-one onto
 $D_{|c_1|r}(0)$. Furthermore, if
 \[
 d = \max\{i\geq 1:|c_i|{r}^i=|c_1| r\},
 \]
 then $h$ maps the  closed disc $\overline{D}_{r}(\alpha )$ onto
 $\overline{D}_{|c_1|r}(0)$ exactly $d$-to-1 (counting
 multiplicity).

 \item Suppose that $h$ converges on the rational open disc
  $D_R(\alpha )$ (but not necessarily on the sphere $S_R(0)$).
  Let $0<r\leq R$ and suppose that
  \[
   |c_i|r^i \leq |c_1|r\quad \text{ for all } i\geq 2 .
  \]
  Then, $h$ maps $D_{r}(\alpha )$ one-to-one
  onto $D_{|c_1|r}(0)$.

\end{enumerate}

\end{proposition}

As a consequence of Proposition \ref{proposition-discdegree}, $f$
satisfies the following Lipschitz condition.

\begin{proposition}[Lemma 2.7 \cite{Benedetto:2003a}]\label{proposition lipschitz} Let $f$ be a
non-constant power series defined on a disc $U\subset K$ of radius
$r>0$, and suppose that $f(U)$ is a disc of radius $s>0$. Then for
any $x,y\in U$,
\[|f(x)-f(y)|\leq \frac{s}{r}|x-y|.\]
\end{proposition}

Also note the following non-Archimedean analogue of the Complex
Koebe $1/4$-Theorem.

\begin{proposition}[Non-Archimedean Koebe 1-Theorem, Lemma 3.5 \cite{Benedetto:2003a}]\label{proposition koebe}
Let $K$ be algebraically closed. Let $f$ be a power series over
$K$ which is convergent and one-to-one on a disc $U\subset K$,
with $0\in U$. Suppose that $f(0)=0$ and $f'(0)=1$. Then $f(U)=U$.
\end{proposition}

If $f$ and $U$ satisfy the condition of the Keoebe theorem, then
by the Lipschitz condition in Proposition \ref{proposition
lipschitz}, $f:U \to U $ is not only bijective but also isometric.
We have the following lemma.

\begin{lemma}\label{lemma indiff fixed point disc}
Let $K$ be algebraically closed. Let $f$ be a power series over
$K$, which converges and is one-to-one on a disc $U\subset K$.
Suppose that there is an element $x_0\in U$ such that $f(x_0)=x_0$
and $|f'(x_0)|=1$. Then $f:U\to U$ is bijective and isometric.
\end{lemma}
\begin{proof}
First, assume that $U$ is rational closed. Consider the function
$h(x)=f(x)-x_0$. By definition, $h$ is also one-to-one on $U$.
Moreover, $h(x_0)=0$ and $|h'(x_0)|=1$. Thus, in view of
Proposition \ref{proposition-discdegree}, $h(U)$ is a rational
closed disc and the radius of $h(U)$ is the same as that of $U$.
It follows that $f(U)$ is rational closed and that the radius of
$f(U)$ is the same as that of $U$. Because both $U$ and $f(U)$
contain $x_0$, we have $f(U)=U$. The remaining case is when $U$ is
open. Write $U$ as the union $\cup U_i$ of rational closed discs
containing $x_0$. Then $f(U)=\cup f(U_i)=\cup U_i=U$.

Next, we show that $f:U\to U$ is isometric. As the radius of
$h(U)$ is the same as that of $U$, we have by Proposition
\ref{proposition lipschitz} that $|h(x)-h(y)|\leq |x-y|$. On the
other hand, since $h:U\to h(U)$ is bijective, we have
\[
|x-y|=|h^{-1}\circ h(x)- h^{-1}\circ h(y)|\leq |h(x)-h(y)|.
\]
Consequently, $|h(x)-h(y)|=|x-y|$ so that $h:U\to h(U)$, and hence
$f:U\to U$, is isometric.
\end{proof}

In fact, a power series $f$ over a complete non-Archimedean field
$K$, is always one-to-one (and hence isometric) on some non-empty
disc about an indifferent fixed point $x_0\in K$. This is a
consequence of the local invertibility theorem
\cite{Schikhof:1984}. The maximal such disc is given by the
following lemma.

\begin{lemma}\label{lemma  f one-to-one}
Let $K$ be algebraically closed. Let $f\in K[[x]]$ be convergent
on some non-empty disc about $x_0\in K$. Suppose that $f(x_0)=x_0$
and $|f'(x_0)|=1$, and write
\[
f(x)=x_0+\lambda (x-x_0)+\sum_{i\geq 2}a_i(x-x_0)^i, \quad
a=\sup_{i\geq 2}|a_i|^{1/(i-1)}.
\]
Let $M$ be the largest disc, with $x_0\in M$, such that $f:M\to M$
is bijective (and hence isometric). Then $M=D_{1/a}(x_0)$ if
either $\max_{i\geq 2 }|a_i|^{1/(i-1)}$ is attained (as for polynomials) or $f$
diverges on $S_{1/a}(x_0)$. Otherwise,
$M=\overline{D}_{1/a}(x_0)$.
\end{lemma}
\begin{proof}
Because $f$ is convergent, we must have
\[
a=\sup_{i\geq 2}|a_i|^{1/(i-1)}<\infty.
\]
Moreover, $f$ is certainly convergent on the open disc
$D_{1/a}(x_0)$. As in the proof of Lemma \ref{lemma indiff fixed
point disc}, it is sufficient to consider the mapping properties
of the map $h(x)=f(x)-x_0$.

First, in view of Proposition \ref{proposition one-to-one},
$h:D_{1/a}(x_0)\to D_{1/a}(0)$ is one-to-one, since by definition
\[
|a_i|(1/a)^i\leq 1/a= |a_1|(1/a).
\]

Second, if $h$ converges on the closed disc
$\overline{D}_{1/a}(x_0)$ and $\max_{i\geq 2 }|a_i|^{1/(i-1)}$ is
attained for some $i\geq 2$, then
\[d = \max\{i\geq
1:|a_i|{(1/a)}^i=|a_1| (1/a)\}\geq 2.
\]
By Proposition \ref{proposition one-to-one}, $h$ is not one-to-one
on $\overline{D}_{1/a}(x_0)$.

Third, if $h$ converges on the closed disc
$\overline{D}_{1/a}(x_0)$ and $\max_{i\geq 2 }|a_i|^{1/(i-1)}$ is
never attained. Then,
\[
|a_i|(1/a)^{i}<1/a=|a_1|(1/a),
\]
for all $i\geq 2 $, so that $d=1$. In other words,
$h:\overline{D}_{1/a}(x_0)\to \overline{D}_{1/a}(0)$ is one-to-one
in this case. However, $h$ cannot be one-to-one on any (rational)
disc strictly containing $\overline{D}_{1/a}(x_0)$; if
$r<a=\sup_{i\geq 2}|a_i|^{1/(i-1)}$, then $|a_N|^{1/(N-1)}\geq r$
and hence
\[
|a_N|(1/r)^{N}\geq 1/r=|a_1|(1/r),
\]
for some $N\geq 2$. This completes the proof.
\end{proof}



It follows from the proof above that if $f$ converges on the
sphere $S_{1/a}(x_0)$ but fails to be one-to-one there, then there
is a point $x\in S_{1/a}(x_0)$ such that $f(x)=x_0=f(x_0)$. This
is always the case when $f$ is a polynomial.

That $f$ may diverge on $S_{1/a}(x_0)$ follows since, for example, the
power series $f(x)=\lambda x + \sum_{i=2}^{\infty}(a_2)^{i-1}x^i$
converges if and only if $|x|<1/|a_2|=1/a$.

Furthermore, for every $x\in M$, $|f(x)-x_0|=|x-x_0|$ and hence
all spheres in $M$ are invariant under $f$.

\begin{remark}\label{remark Lemma bijective}
Recall that the discs $D_{1/a}(0)$ and $\overline{D}_{1/a}(0)$ are
rational if and only if $a=\sup_{i\geq 2}|a_i|^{1/(i-1)}\in |K|$.
If the maximum $a=\max_{i\geq 2}|a_i|^{1/(i-1)}$ exists, and $K$
is algebraically closed,  then $a\in|K|$. This is always the case
if $f$ is a polynomial. If $f$ is not a polynomial and the maximum
fails to exist we may have $\sup_{i\geq 2} |a_i|^{1/(i-1)}\notin
|K|$. Let $K=\mathbb{C}_p$. Let $\beta $ be an irrational number
and let $p_n/q_n$ be the $n$-th convergent of the continued
fraction expansion of $\beta$. Let the sequence $\{a_i\in
\mathbb{Q}_p\}_{i\geq 2}$ satisfy
\[
|a_i|= \left \{
\begin{array}{ll}
p^{p_n}, & \textrm{if \quad $i-1=q_n$ and $p_n/q_n<\beta $},\\
0, & \textrm{otherwise}.
\end{array}\right.
\]
Then,
\[
\sup_{i\geq 2} |a_i|^{1/(i-1)}=p^{\beta}\notin
|K|=\{p^r:r\in\mathbb{Q} \}\cup \{0\}.
\]
\end{remark}

For more information on non-Archimedean power series the reader
can consult
\cite{Schikhof:1984}.
From a dynamical point of view,  the paper \cite{Benedetto:2003a}
contains many useful results on non-Archimedean analogues of
complex analytic mapping theorems relevant for dynamics.

\subsection{Linearization discs}

The results above have some important implications for linearization 
discs. We use the following definition of a linearization disc. Let $K$
be a complete non-Archimedean field. Suppose that $f\in K[[x]]$
has an indifferent fixed point $x_0\in K$, with multiplier
$\lambda =f'(x_0)$, not a root of unity. By
\cite{Herman/Yoccoz:1981}, there is a unique formal power series
solution $g$, with $g(x_0)=0$ and $g'(x_0)=1$, to the following
form of the Schr\"{o}der functional equation
\[
g\circ f(x)=\lambda g(x).
\]
If the formal solution $g$ converges on some non-empty disc about
$x_0$, then the corresponding \emph{linearization disc} of $f$ about
$x_0$, denoted by $\Delta _f(x_0)$, is defined as the largest disc
$U\subset K$, with $x_0\in U$, such that the Schr\"{o}der
functional equation holds for all $x\in U$, and $g$ converges and
is one-to-one on $U$. We will often refer to $g$ as the
\emph{conjugacy function}.

This notion of a linearization disc is well-defined since
%
%
by the proof of  Lemma \ref{lemma f one-to-one}, there always
exist a largest disc on which $g$ is one-to-one (provided that $g$
is convergent). Also note that by the non-Archimedean Siegel
theorem \cite{Herman/Yoccoz:1981} and the fact that $\mathbb{C}_p$
is of characteristic zero, the formal solution $g$ always
converges if the state space $K=\mathbb{C}_p$.

As one might expect from previous results, both $f$ and the
conjugacy $g$ turn out to be one-to-one and isometric on a
non-Archimedean linearization disc.

\begin{lemma}\label{lemma linearization disc isometry}
Let $K$ be algebraically closed. Suppose that $f\in K[[x]]$ has a
linearization disc $\Delta_f(x_0)$ about $x_0\in K$. Let $g$, with
$g(x_0)=0$ and $g'(x_0)=1$, be the corresponding conjugacy
function. Then, both $g:\Delta_f(x_0)\to g(\Delta_f(x_0))$ and
$f:\Delta_f(x_0)\to \Delta_f(x_0)$ are bijective and isometric. In
particular,
if $x_0=0$, then $g(\Delta_f(x_0))=\Delta_f(x_0)$. Furthermore,
$\Delta_f(x_0)\subseteq M\subseteq \overline{D}_{1/a}(x_0)$, where
$M$ and $a$  are defined as in Lemma \ref{lemma  f one-to-one}.
\end{lemma}
\begin{proof}
By the conjugacy relation $g\circ f(x)=\lambda g(x)$ and the fact
that the map $g:\Delta _f(x_0)\to g(\Delta _f(x_0))$ is one-to-one,
$f:\Delta _f(x_0)\to \Delta _f(x_0)$ is also one-to-one and hence
bijective and isometric by Lemma \ref{lemma indiff fixed point
disc}.

Recall that $g(x_0)=0$ and $g'(x_0)=1$. That $g:\Delta_f(x_0)\to
g(\Delta_f(x_0))$ is bijective and isometric then follows by same
arguments as those applied to $h$, in the proof of Lemma
\ref{lemma indiff fixed point disc}.
\end{proof}

As a consequence, the radius of a linearization disc $\Delta_f(x_0)$ is
equal to to that of $g(\Delta_f(x_0))$. In particular, the radius
of a linearization disc is independent of the location of the fixed point
$x_0$. Therefore, we shall, without loss of generality, henceforth
assume that $x_0=0$.

The forthcoming sections are very much devoted to estimates of the
maximal disc on which $g$ is one-to-one. Before dealing with this
more delicate problem, note the following remark.

\begin{remark}
All the results in this and the previous section, except for Proposition
\ref{proposition-discdegree}, hold also in the case that $K$ is
not algebraically closed, with the modification that the mappings
are are one-to-one but not necessarily surjective. However, with
certain restrictions on the multiplier $\lambda$, e.g.\@
$f:\Delta_f(x_0)\cap \mathbb{Q}_p \to \Delta_f(x_0)\cap
\mathbb{Q}_p$ may also be surjective, see Corollary \ref{corollary
surjective in Qp}. In fact, as stated in Theorem
\ref{theorem-minimalitypreserved}, if $f$ is transitive on a
compact subset $X$ of a linearization disc, then $f(X)=X$.
\end{remark}

\subsection{The formal solution}
As noted in the previous section, we may, without loss of
generality,  assume that $f$ has its fixed point at the origin,
and that $f\in \mathcal{F}_{\lambda,a}$, as defined below. Let
$\lambda\in \mathbb{C}_p$ be such that
\begin{equation}
|\lambda|=1, \quad\textrm{but } \lambda ^n\neq 1,\quad  \forall
n\geq 1,
\end{equation}
and let $a$ be a real number. We shall associate with the pair
$(\lambda,a)$ a family $\mathcal{F}_{\lambda,a}$ of power series
defined by
\begin{equation}
\mathcal{F}_{\lambda,a}:=\left\{\lambda x
+\sum a_ix^i\in\mathbb{C}_p[[x]]:a=\sup_{i\geq 2}|a_i|^{1/(i-1)}\right\}.
\end{equation}
It follows that each $f\in \mathcal{F}_{\lambda,a}$ is convergent
on $D_{1/a}(0)$, and by Lemma \ref{lemma  f one-to-one}
$f:D_{1/a}(0)\to D_{1/a}(0)$ is bijective and isometric.

As $\mathbb{C}_p$ is of characteristic zero, we may, by the
non-Archimedean Siegel theorem \cite{Herman/Yoccoz:1981},
associate with $f$ a unique convergent power series solution $g$
to the Scr\"{o}der functional equation, of the form
\[
g(x)=x + \sum_{k\geq 2}b_kx^k,
\]
and a corresponding linearization disc about the origin
\[
\Delta_f:=\Delta _f(0).
\]
Recall that by Lemma \ref{lemma linearization disc isometry}, since
$x_0=0$, the linearization disc $\Delta_f$ is the largest disc $U\subset
\mathbb{C}_p$ about the origin such that the full conjugacy
$g\circ f \circ g^{-1}(x)=\lambda x$ holds for all $x\in U$.

\newpage
Given $f\in \mathcal{F}_{\lambda,a}$, Lemma \ref{lemma linearization disc
isometry} yields the following concerning $\Delta_f$.

\begin{lemma}\label{lemma upper bound Siegel and isometry}
Let $f\in \mathcal{F}_{\lambda,a}$. Then $f$ has a linearization disc
$\Delta_f$ about the origin in $\mathbb{C}_p$. Let $g$, with
$g(0)=0$ and $g'(0)=1$, be the corresponding conjugacy function.
Then, the following two statements hold:
\begin{enumerate}[1)]

   \item Both $g:\Delta_f\to \Delta_f$ and $f:\Delta_f\to \Delta_f$ are bijective
   and isometric.

   \item $\Delta_f\subseteq \overline{D}_{1/a}(0)$. If $a=\max_{i\geq 2}
   |a_i|^{1/(i-1)}$ or $f$ diverges on the sphere $S_{1/a}(0)$, then $\Delta_f\subseteq D_{1/a}(0)$.

\end{enumerate}
\end{lemma}

Our results on lower bounds for linearization discs are based on the
following lemma.

\begin{lemma}\label{lemma bk estimate indiff}
Let $f\in \mathcal{F}_{\lambda,a}$. Then, the coefficients of the
conjugacy function $g$ satisfy
\begin{equation}\label{b_k estimate by prod 1-lambda n}
|b_k|\leq \left( \prod_{n=1}^{k-1}|1-\lambda ^n| \right
)^{-1}a^{k-1},
\end{equation}
for all $k\geq 2$.
\end{lemma}

\begin{proof} The coefficients of the conjugacy $g$ must
satisfy the recurrence relation
\begin{equation}\label{bk-equation}
b_k=\frac{1}{\lambda (1-\lambda^{k-1})}\sum_{l=1}^{k-1}b_l%
(\sum\frac{l!}{\alpha_1!\cdot ...\cdot
\alpha_k!}a_1^{\alpha_1}\cdot ...\cdot a_k^{\alpha_k})
\end{equation}
where $\alpha _1,\alpha _2,\dots,\alpha _k$ are nonnegative
integer solutions of
\begin{equation}\label{index-equations}
   \left\{\begin{array}{ll}
            \alpha_1+...+\alpha_k=l,\\
            \alpha_1+2\alpha_2...+k\alpha_k=k,\\
            1\leq l\leq k-1.
        \end{array}
   \right.
\end{equation}
Note that the factorial factors
$l!/\alpha_1!\cdot\dots\cdot\alpha_k!$ are always integers and
thus of modulus less than or equal to $1$. Also recall that
$|a_i|\leq a^{i-1}$.  It follows that
\[
|b_k|\leq \left( \prod_{n=1}^{k-1}|1-\lambda ^n| \right )^{-1}
a^{\alpha},
\]
for some integer $\alpha$. In view of equation
(\ref{index-equations}) we have
\[
\sum_{i=2}^{k}(i-1)\alpha_i=k-l.
\]
Consequently, since $|a_i|\leq a^{i-1}$, we obtain
\begin{equation}\label{estimate prod a_i}
\prod_{i=2}^k|a_i|^{\alpha_i}\leq
\prod_{i=2}^{k}a^{(i-1)\alpha_i}=a^{k-l}.
\end{equation}
Now we use induction over $k$. By definition $b_1=1$ and,
according to the recursion formula (\ref{bk-equation}), $|b_2|\leq
|1-\lambda |^{-1}|a_2|\leq |1-\lambda |^{-1}|a|$. Suppose that
\[
|b_l|\leq \left( \prod_{n=1}^{l-1}|1-\lambda ^n| \right
)^{-1}a^{l-1}
\]
for all $l<k$. Then
\[
|b_k|\leq \left( \prod_{n=1}^{k-1}|1-\lambda ^n| \right
)^{-1}a^{l-1}\max\left\{ \prod_{i=2}^k|a_i|^{\alpha_i}\right \},
\]
and the lemma follows by the estimate (\ref{estimate prod a_i}).
\end{proof}

In the following sections we show how to calculate the distance
$|1-\lambda ^n |$ for an arbitrary integer $n\geq 1$. Applying
Proposition \ref{proposition one-to-one} to the estimate in the
above lemma we can then estimate the disc on which the conjugacy
function $g$ is one-to-one.

\subsection{Geometry of the unit sphere and the roots of unity}\label{section geometry and 
roots of uni ty C p}

Let $\Gamma$ be the group of all roots of unity in $\mathbb{C}_p$.
It has the important subgroup $\Gamma_u$ ($u$ for unramified),
given by
\begin{equation}\label{Gammau}
 \Gamma_u=\{\xi\in\mathbb{C}_p:\textrm{ } \xi ^m=1\textrm{
for some $m$ not divisible by $p$}\}.
\end{equation}

\begin{proposition}[$\Gamma_u$ and the unit sphere]\label{proposition gamma u}
The unit sphere $S_1(0)$ in $\mathbb{C}_p$ decomposes into the
disjoint union
\[
S_1(0)=\cup _{\xi \in \Gamma_u}D_1(\xi).
\]
In particular, $\Gamma_u\cap D_1(1)=\{ 1\}$ and consequently
$|1-\xi|=1$ for all $\xi \neq 1$. To each $\lambda \in S_1(0)$
there is a unique $\xi\in \Gamma_u$ and $h\in D_1(1)$ such that
$\lambda =\xi h$.
\end{proposition}
\begin{proof}
See \cite[p. 103]{Schikhof:1984}.
\end{proof}
Let us note that $\Gamma_u$ is isomorphic to the multiplicative
subgroup $\overline{\mathbb{F}}_p\setminus\{0\}$ of the residue
field in $\mathbb{C}_p$.

Another important subgroup of $\Gamma$ is $\Gamma_r$ ($r$ for
ramified), given by
\begin{equation}\label{Gammar}
\Gamma_r=\{\zeta\in\mathbb{C}_p:\textrm{ }\zeta^{p^s}=1 \textrm{
for some integer $s\geq 0$}\}.
\end{equation}
By elementary group theory $\Gamma _u\cap\Gamma _r =\{1\}$ and
$\Gamma =\Gamma _u\times\Gamma _r$, see e.g.\@\@ the paper \cite[p.
103]{Schikhof:1984}). Most importantly, the $p^s$th roots of unity
in $\Gamma_r $ are located on spheres about the point $x=1$ of
radius $R(t)$, where
\begin{equation}\label{definition R(s) extended}
R(t):=\left\{\begin{array}{ll}
0, & \textrm{if \quad $t=0$,}\\
p^{-\frac{1}{p^{t-1}(p-1)}}, & \textrm{if \quad $t\geq 1$.}
\end{array}\right.
\end{equation}
This fact is fundamental for our estimates of linearization discs.

\begin{proposition}[The geometry of $\Gamma_r$]\label{proposition gamma r}
$\Gamma_r\subset D_1(1)$. If $\zeta\in\Gamma_r$ is a primitive
$p^s$th root of unity for some $s\geq 0$, then
\[
|1-\zeta|=R(s).
\]
Moreover, if $\zeta_1,\zeta_2\in S_{R(s)}(1)$ and
$\zeta_1\neq\zeta_2$, then
\[
|\zeta_1-\zeta_2|=R(s).
\]
Furthermore, for each $s\geq 1$, there are $p^s-p^{s-1}$ different
roots of unity on the sphere $S_{R(s)}(1)$.
\end{proposition}
\begin{proof}
See  for example \cite{Escassut:1995}.
\end{proof}
As a consequence we have the following lemma.

\begin{lemma}\label{lemma closest root of unity}
Let $\lambda\in D_1(1)$ be not a root of unity. Then, there exist
$\alpha\in\Gamma_r$ such that $|\alpha-\lambda|\leq |\gamma
-\lambda|$, for all $\gamma\in\Gamma$. Furthermore, if
$|1-\lambda|\neq R(s)$ for every $s\geq 0$. Then, $\alpha =1$ is
the only root of unity with this property.
\end{lemma}

\subsection{Transitivity of the multiplier map}
In this section, $K$ is a finite extension of $\mathbb{Q}_p$ of degree $[K:\mathbb{Q}_p]=e\cdot f$, with  ramification index $e$, residue field degree $f=[k:\mathbb{F}_p]$, and  uniformizer $\pi$.

Let $\lambda\in \mathbb{C}_p$ be an element on the unit sphere
$S_1(0)$. We are concerned with calculating the distance
\[
|1-\lambda ^n|
\]
for each integer $n\geq 1$.  In view of Proposition \ref{proposition gamma u}
there is an integer $m$, not divisible by $p$, such that $\lambda
= \xi h$ for some $m$th root of unity $\xi\in\Gamma_u$ and $h\in
D_1(1)$. In other words, the following integer exists
\begin{equation}\label{definition of m}
m=m(\lambda):=\min\{n\in \mathbb{Z}:n\geq 1, |1-\lambda ^n|<1 \}.
\end{equation}
Also note that, since the residue field is of characteristic
$p>0$, $m$ is not divisible by $p$. In fact, if $\lambda $ belongs
to a finite algebraic extension $K$ of $\mathbb{Q}_p$, then 
\[
1\leq m\leq p^f-1, 
\]
where $p^f$ is the number of elements in the
residue field $k$ of $K$. In particular, if
$\lambda\in\mathbb{Q}_p$ we have $1\leq m\leq p-1$.

\begin{lemma}\label{lemma mod pi}
Let $K$ be a finite extension of $\mathbb{Q}_p$, with uniformizer $\pi$ and ramification index $e$.
Let $x \in K$, and $\delta =\min \{1+e,p \}$. If $x \equiv 1 \mod \pi$, then we have $x ^p \equiv 1\mod \pi ^{\delta}$. Moreover, if $x \not\equiv 1\mod \pi ^{2}$, 
then $x^l\not\equiv 1\mod \pi ^2$, $1\leq l\leq p-1$. 
\end{lemma}
\begin{proof}
First, suppose $x\in 1+O(\pi)$. Then
\[
x^p\in(1+O(\pi))^p=1+pO(\pi)+\sum_{k=2}^{p-1}\binom{p}{k}O(\pi^k)
+O(\pi^p)=1+O(\pi^{\delta}),
\]
where the last equality follows from the fact that $|p|=|\pi^{e}|$.

Second, suppose $x\in1+ a_1\pi +O(\pi^2)$, $a_1\neq 0$, and  $1\leq l\leq p-1$. Then 
\[
x^l\in(1+a_1\pi +O(\pi^2))^l=
\sum_{k=0}^{l}\binom{l}{k}(a_1\pi +O(\pi^2))^k
=1+la_1\pi +O(\pi^2).
\]
\end{proof}
Using the terminology from \cite{Pettigrew/Roberts/Vivaldi:2001},
we say that an element $\lambda $ on the unit sphere in $K $ is \emph{primitive} if 
$m=p^f-1$,
  and that $\lambda $ is \emph{maximal}  if in addition, 
$\lambda^{p^f-1}\not\equiv 1\mod \pi^2$ (so that  $|1-\lambda ^m|=|\pi |=p^{-1/e}$).  

\begin{remark}
Note that \cite{Pettigrew/Roberts/Vivaldi:2001} only consider the case $\lambda\in \mathbb{Q}_p$ for odd primes, whereas we also consider extensions of $\mathbb{Q}_p$, including the case $p=2$. Therefore, a maximal $\lambda$ does not always give a dense orbit as explained below.  
\end{remark}

It follows from Lemma \ref{lemma mod pi} that if $\lambda$ is maximal,
then the multiplication map $T_{\lambda}:\lambda \mapsto \lambda x$
has cycle length $p^f-1$ modulo $\pi$, and length $p(p^f-1)$ modulo $\pi ^2, \dots$, $\pi^{\delta}$.
As a consequence, $T_{\lambda}$ cannot act as a permutation modulo $\pi^{3}$ if $e\geq 2$.
In particular, $T_{\lambda}$ cannot have a dense orbit on the unit sphere in $K$ in this case.
We have the following Lemma.

\begin{lemma}\label{lemma transitivity multiplier}
Let $K$ be a proper extension of $\mathbb{Q}_p$ and let $\lambda\in K$, with $|\lambda |=1$. Then the map $T_{\lambda}:$ $x\mapsto\lambda x$ , 
 cannot be transitive on any sphere about the origin in $K$. Furthermore, if $\lambda \in\mathbb{Q}_p$, then $T_{\lambda}$ is transitive on a sphere about the origin (in fact all spheres inside the unit disc) in $\mathbb{Q}_p$ if and only if $\lambda$ is maximal and $p$ is odd.
\end{lemma}
\begin{proof}
Let $[K:\mathbb{Q}_p]=e\cdot f\geq 2$. First, suppose that $f>1$. By Lemma \ref{lemma mod pi}
$\lambda ^{p(p^f-1)}\equiv 1 \mod \pi^2$. Consequently, $T_{\lambda}$ cannot act as a permutation
on all $p^{f}(p^f-1)$ elements in the group of units modulo $\pi^2$.

Second, suppose $f=1$ but $e>1$.  By Lemma \ref{lemma mod pi}
$\lambda ^{p(p-1)}\equiv 1 \mod \pi^3$. Consequently, $T_{\lambda}$ cannot act as a permutation
on all $p^{2}(p-1)$ units modulo $\pi^3$.

The last statement of the lemma follows from the fact that a integer $\lambda$ is a generator of
the group of
units modulo $p^k$ for every integer $k\geq 2$ if and only if $\lambda$ is a generator modulo $p^2$, 
which happen if and only if  $\lambda $ is maximal and $p$ is odd. 

\end{proof}
 
\subsubsection{Non-transitivity in characteristic $p$} 

To motivate Theorem \ref{theorem non-archimedean ergodic disc}, we also show that the multiplier map can never be transitive on a whole sphere for  fields of prime characteristics.  
\begin{proposition}\label{proposition mod pi char p}
Let $K$ be a locally compact field of prime characteristc, with uniformizer $\pi$. 
If $x\in K$ and $x \equiv 1 \mod \pi$, then $x ^{p^n} \equiv 1\mod \pi ^{p^n}$
for all integers $n\geq 0$.  
\end{proposition}
\begin{proof}
Suppose $x\in 1+O(\pi)$. Then
\[
x^p\in(1+O(\pi))^{p^n}=1+p^{n}O(\pi)+\sum_{i=2}^{p-1}\binom{p^{n}}{i}O(\pi^i)
+O(\pi^{p^n})=1+O(\pi^{p^n}).
\]
\end{proof}

\begin{lemma}\label{lemma non transitivity multiplier char p}
Let $K$ be a locally compact field of prime characteristic and let $\lambda\in K$, with $|\lambda |=1$. Then, the map $T_{\lambda}$: $x\mapsto\lambda x$ , 
 cannot be transitive on any sphere about the origin in $K$.
\end{lemma}
\begin{proof}
By local compactness of $K$, there is a uniformizer $\pi\in K$, and hence $T_{\lambda}$ is transitive on the unit sphere if and only if it is transitive on the group of units modulo $\pi^n$ for every integer $n\geq 2$. 
In particular, $T_{\lambda}$ has to be transitive modulo $\pi^{p^2}$ which is impossible by the
following arguments. As $K$ is locally compact, the residue field $k$ is finite. Let $c$ be the cardinality of $k$. By definition $c\geq p$, and hence there are $(c-1)p^{p^2-1}>p^2$ units modulo  $\pi^{p^2}$. On the other hand, by Proposition \ref{proposition mod pi char p}, $\lambda ^{p^2} \equiv 1\mod \pi ^{p^2}$. Consequently, $T_{\lambda}$ cannot be transitive modulo $\pi^{p^2}$ and hence not on the unit sphere or any other sphere about the origin in $K$.

\end{proof}

\subsection{Arithmetic of the multiplier}

\begin{lemma}\label{lemma distance 1-lambda m char 0,p}
Let $\lambda \in \mathbb{C}_p$ be an element on the unit sphere,
but not a root of unity. Let $m$ be defined by (\ref{definition of
m}),  and let $s$ be the integer for which 
$ R(s)\leq |1-\lambda ^m|< R(s+1) $.
Then, the following three statements hold:  

\begin{enumerate}

\item If $m$ does not divide $n$, then $|1-\lambda^n |=1$.

\item If $m$ is a divisor of $n$ and  $0\leq s\leq \nu(n)$ we have
\begin{equation*}
\left |1 -\lambda ^n \right |= \left\{\begin{array}{ll}
|n|p^s|1-\lambda ^m|^{p^s}, & \textrm{if  $ R(s)<|1-\lambda ^m|< R(s+1) $,}\\
|n|p^s|1-\lambda ^m|^{p^s-1}|\alpha -\lambda ^m|, & \textrm{if  $|1-\lambda ^m|= R(s) $.}
\end{array}\right.
\end{equation*}
Here $\alpha\in \Gamma _r$ is chosen so that 
$|\alpha-\lambda^m|\leq |\gamma -\lambda^m|$, for all
$\gamma\in\Gamma$.

\item If  $m$ is a divisor of $n$ and $s>\nu(n)$ so that $ |1-\lambda ^m|> R(\nu (n))$, then 
\[
\left |1 -\lambda ^n \right |=|1-\lambda ^m|^{p^{\nu(n)}}.
\]

\end{enumerate}

\end{lemma}

\begin{remark}
In the second statement of the lemma, $|n|p^s\leq 1$ since we assume that $s\leq \nu(n)$ in this case.
\end{remark}

\begin{remark}\label{remark lemma distance 1-lambda m char 0,p}
By Lemma \ref{lemma closest root of unity}, we may have $|\alpha
-\lambda ^m|< |1-\lambda ^m|$ only if $\lambda ^m$ belong to the
same sphere about $1$ as $\alpha\in \Gamma _r$. In all other cases
we may choose $\alpha = 1$. In particular, we always have $|\alpha -\lambda
^m|\leq |1-\lambda ^m|$.
\end{remark}

\begin{remark}
If $\lambda \in \mathbb{C}_2$ and $|1-\lambda
^m|=2^{-1/(2-1)}=2^{-1}$, then the `closest' root of unity $\alpha
=-1=\sum_{k=0}^{\infty}2^k$.
\end{remark}

\begin{proof}
It is enough to consider the case $m=1$.
This proof is based on the factorization of the polynomial
$\lambda ^n-1$ from which we find
\begin{equation}\label{basicfactorequation}
\left |\lambda ^n - 1 \right |=\prod _{\theta ^n=1}\left |\lambda
- \theta \right |.
\end{equation}
As noted in Section \ref{section geometry and roots of uni ty C p}
the roots of unity in $\mathbb{C}_p$ is the direct product $\Gamma
=\Gamma _u\times\Gamma _r$. This representation enables us to
write $\left |\lambda ^n - 1 \right |$ in the form
\begin{equation}\label{splitedfactorequation}
\left |\lambda ^n - 1 \right |= \left |\lambda  - 1 \right | \prod
_{\zeta ^n=1}\left |\lambda - \zeta \right | \prod _{\xi
^n=1}\left |\lambda - \xi \right | \prod _{(\zeta\xi )^n=1}\left
|\lambda - \zeta\xi \right |,
\end{equation}
where $\zeta\in \Gamma_r\setminus \{1\}$ and
$\xi\in\Gamma_u\setminus\{1\}$. Recall that we assume that
$\lambda\in D_1(1)$. In view of Proposition \ref{proposition gamma u},
$\xi,(\zeta\xi)\notin D_1(1)$ and consequently $|\lambda
-\xi|=|\lambda - \zeta\xi|=1$. Moreover, for $n=ap^{\nu(n)}$, we
have that $\zeta^n=1$ if and only if $\zeta^{p^{\nu(n)}}=1$. It
follows that (\ref{splitedfactorequation}) can be reduced to
\begin{equation}\label{simplifiedfactorequation}
\left |\lambda ^n - 1 \right |= \left |\lambda - 1 \right | \prod
_{\zeta ^{p^{\nu(n)}}=1}\left |\lambda - \zeta \right |.
\end{equation}
If $p$ does not divide $n$ so that $\nu(n)=0$, then $|1-\lambda
^n|=|1-\lambda|$ as required.

In the remaining cases $\nu(n) \geq 1$ and we have to take the
factors $|\lambda-\zeta|$ into account. Note that $|\lambda
-\zeta|=|(\lambda -1) +(1-\zeta)|$ and by ultrametricity
\begin{equation}\label{distance to roots}
   \left |\lambda - \zeta \right |=\max\{\left |\lambda - 1 \right |,\left |1-\zeta \right |\},
\end{equation}
if  $\left |\lambda - 1 \right |\neq\left |1 -\zeta \right |$.
Thus we can compute (\ref{simplifiedfactorequation}) by counting
the number of roots $\zeta$ that are closer and farther to $1$
compared to $\lambda $, respectively.

Recall the following facts from  Proposition \ref{proposition gamma r}. If
$\zeta\in\Gamma_r\setminus\{1\}$ is a \emph{primitive} $p^s$th
root of unity for some $s\geq 1$, then $\left |1 - \zeta \right
|=p^{-r_s}$ where $r_s=1/(p^s-p^{s-1})$. The sphere
$S_{p^{-r_s}}(1)$ contains $p^s-p^{s-1}=1/r_s$ roots of unity.
Note that these spheres have radii ordered as
$p^{-r_1}<p^{-r_2}<...<1$.

Now we consider the case $|1-\lambda|< p^{-r_1}$. In this case
$\lambda$ is closer to $1$ than any of the roots $\zeta$ and
therefore, in view of (\ref{distance to roots}), $|\lambda - \zeta
|=|1 -\zeta |$ for every $\zeta\in\Gamma_r\setminus\{1\}$. From
(\ref{simplifiedfactorequation}) we thus have that
\begin{equation*}
\left |\lambda ^n - 1 \right |= \left |\lambda - 1 \right
|(p^{-r_1})^{1/{r_1}}\cdot ...\cdot
(p^{-r_{\nu(n)}})^{1/{r_{\nu(n)}}}=|\lambda -1|p^{-{\nu(n)}}
\end{equation*}
as required.

Now we consider the case $p^{-r_{s}}<|\lambda -1|<p^{-r_{s+1}}$,
$1\leq s\leq \nu(n)$. We have in view of (\ref{distance to roots})
that $|\lambda-\zeta|=|\lambda-1|$ for all $\zeta$ such that
$|1-\zeta|<|\lambda -1|$. This is the case for $p^s$ roots. All
other roots are further from $1$ than $\lambda $ is. For these
roots $|\lambda -\zeta|=|1-\zeta|$. Hence, the right-hand side of
(\ref{simplifiedfactorequation}) becomes
\begin{equation}\label{distance-rs+1<r<rs}
|\lambda -1|^{p^s} (p^{-r_{s+1}})^{1/{r_{s+1}}}\cdot ...\cdot
(p^{-r_{\nu(n)}})^{1/{r_{\nu(n)}}}=|1-\lambda|^{p^s}p^{-({\nu(n)}-s)}
\end{equation}
as required.

Now we consider the case $|\lambda -1|=p^{-r_s}$ for some $1\leq
s\leq \nu(n)$. Let $\alpha\in S_{p^{-r_s}}(1)$ be such that
$|\alpha -\lambda| \leq |\zeta -\lambda|$ for all
$\zeta\in\Gamma_r$. Note that $|\zeta -\lambda|\leq p^{-r_s}$ for
all $\zeta\in S_{p^{-r_s}}(1)$ and that  $\alpha$ is unique if and
only if $|\lambda - \alpha |<p^{-r_s}$. By Proposition \ref{proposition gamma r}, 
$|\lambda -\zeta|=p^{-r_s}$ for all $\zeta\neq \alpha$ on the
sphere $S_{p^{-r_{s}}}(1)$. For  the right-hand side of
(\ref{simplifiedfactorequation}) we obtain
\begin{equation*}
|\lambda -1|^{p^s-1}|\lambda -\alpha |
(p^{-r_{s+1}})^{1/{r_{s+1}}}\cdot ...\cdot
(p^{-r_{\nu(n)}})^{1/{r_{\nu(n)}}},
\end{equation*}
Consequently,
\begin{equation}\label{distance-r=rs}
|\lambda ^n -1|=|\lambda -1|^{p^s-1}|\lambda-\alpha |p^{-(\nu(n)-s
)}
\end{equation}
as required.

Finally, we consider $|\lambda -1|>p^{-r_{\nu(n)}}$. In this case
$|\lambda -1|>|1-\zeta|$ for all $\zeta$ that are $p^{\nu(n)}$th
roots of unity. Consequently, $|\lambda - \zeta|=|\lambda- 1|$ and
we obtain
\[
|\lambda ^n-1|=|\lambda -1|^{p^{\nu(n)}},
\]
as proposed in the lemma. This completes the proof.
\end{proof}

\section{Estimates of linearization discs}

We will estimate the size of the linearization disc for a power series
$f\in \mathcal{F}_{\lambda,a}$.
The estimates are divided into three different cases according to
the three sections below.

\subsection{Case I}\label{section case 1}

In this section we assume that
\begin{equation}\label{equation lambda case 1}
R(0)<|1-\lambda ^m|<R(1).
\end{equation}
In what follows $\sigma_1$ will be the real number defined by
\begin{equation}\label{definition of rho1}
\sigma_1:=a^{-1}p^{-\frac{1}{m(p-1)}}|1-\lambda ^m|^{\frac{1}{m}}.
\end{equation}

\begin{lemma}\label{lemma case 1}
Suppose $\lambda$ is such that $R(0)<|1-\lambda^m|<R(1)$.
Then,
\begin{equation}\label{prod ineq case 1}
\left( \prod_{n=1}^{k-1}|1-\lambda ^n| \right )^{-1}a^{k-1}\leq
p^{-\frac{1}{p-1}}\sigma_1^{-(k-1)},
\end{equation}
with equality if $(k-1)/m$ is an integer power of $p$.
\end{lemma}
\begin{proof}
By Lemma \ref{lemma distance 1-lambda m char 0,p}
\begin{equation}\label{distance 1- lambda n case 1}
\left |1 -\lambda ^n \right |= \left\{\begin{array}{ll}
1, & \textrm{if \quad $m\nmid n$,}\\
|n||1-\lambda ^m|, & \textrm{if \quad $m\mid n$,}
\end{array}\right.
\end{equation}
in this case. Let $N=\lfloor k-1/m \rfloor$ denote the integer
part of $k-1/m$. Then, by (\ref{distance 1- lambda n case 1})
\begin{equation}\label{product 1- lambda case 1}
\prod_{n=1}^{k-1}|1-\lambda ^n|=\left |N !\right ||1-\lambda
^m|^{N }.
\end{equation}
By Lemma \ref{lemma order of p in factorial}
\[
|N!|\geq p^{-\frac{N-1}{p-1}}\geq p^{-\frac{k-1}{m(p-1)}
+\frac{1}{p-1}},
\]
where each inequality become an equality if $(k-1)/m$ is an
integer power of $p$. It follows by (\ref{product 1- lambda case
1}) that
\[
\left( \prod_{n=1}^{k-1}|1-\lambda ^n| \right )^{-1}a^{k-1}\leq
p^{-\frac{1}{p-1}}\sigma_1^{-(k-1)},
\]
with equality if $(k-1)/m$ is an integer power of $p$.

\end{proof}

 We will prove the following theorem.

\begin{theorem}\label{theorem case 1}
Let $f\in \mathcal{F}_{\lambda,a}$ and suppose $\lambda $ is such
that $R(0)<|1-\lambda^m|<R(1)$.
Then, the linearization disc $\Delta_f\supseteq D_{\sigma_1}(0)$.
Moreover, if the conjugacy function $g$ converges on the closed
disc $\overline{D}_{\sigma_1}(0)$, then $\Delta_f\supseteq
\overline{D}_{\sigma_1}(0)$.
\end{theorem}
\begin{proof}

In view of Lemma \ref{lemma order of p in factorial} and Lemma
\ref{lemma case 1} we have
\[
\left ( \limsup |b_k|^{1/k}  \right)^{-1}\geq \sigma_1.
\]
This implies that $g$ converges on the open disc of radius
$\sigma_1$. Moreover, by Lemma \ref{lemma case 1}
\[
|b_k|\sigma_1^k\leq p^{-\frac{1}{p-1}}\sigma_1
<\sigma_1=|b_1|\sigma_1.
\]
It follows by Proposition \ref{proposition one-to-one} that
$g:D_{\sigma_1}(0)\to D_{\sigma_1}(0)$ is a bijection. The strict
inequality $|b_k|\sigma_1^k<|b_1|\sigma_1$ implies that, if $g$
converges on the closed disc $\overline{D}_{\sigma_1}(0)$, then
$g:\overline{D}_{\sigma_1}(0)\to\overline{D}_{\sigma_1}(0)$ is
bijective. Recall that by Lemma \ref{lemma  f one-to-one}
$f:D_{1/a}(0)\to D_{1/a}(0)$ is a bijection. Moreover,
$1/a>\sigma_1$. Consequently, the linearization disc $\Delta_f$ includes
the disc $D_{\sigma_1}(0)$ or $\overline{D}_{\sigma_1}(0)$,
depending on whether the conjugacy function $g$ converges on the
closed disc $\overline{D}_{\sigma_1}(0)$ or not.
\end{proof}


The theorem has some important consequences for linearization discs in
$\mathbb{Q}_p$. In fact, as we will see below, a linearization disc in
$\mathbb{Q}_p$ may coincide with the maximal disc on which $f$ is
one-to-one and even with the region of convergence of $f$. Recall
that the value group $|\mathbb{Q}_p^*|$ contains only integer
powers of $p$. This implies that if $\lambda\in \mathbb{Q}_p$,
then $|1-\lambda ^m|\leq p^{-1}<p^{-1/(p-1)}$ if $p>2$.
Consequently, Theorem \ref{theorem case 1} applies if $\lambda \in
\mathbb{Q}_p$ for some prime $p>2$. In particular, if $f$ is a
power series over $\mathbb{Q}_p$, we have the following corollary.

\begin{corollary}[linearization discs in $\mathbb{Q}_p$]\label{corollary p-adic linearization disc}
Let $f\in \mathcal{F}_{\lambda,a}\cap\mathbb{Q}_p[[x]]$ for some
odd prime $p$. Let $\Delta _f(0,\mathbb{Q}_p)=\Delta_f\cap \mathbb{Q}_p$ be the corresponding
linearization disc, about the origin, in $\mathbb{Q}_p$. Then,
$\Delta_f(0,\mathbb{Q}_p) \supseteq D_{\sigma_1}(0,\mathbb{Q}_p)$. Moreover, if
the conjugacy function $g$ converges on the closed disc
$\overline{D}_{\sigma_1}(0)$, then the linearization disc 
$\Delta_f(0,\mathbb{Q}_p)\supseteq \overline{D}_{\sigma_1}(0,\mathbb{Q}_p)$.
Furthermore, if $\lambda$ is maximal, 
then, the linearization disc $\Delta_f(0,\mathbb{Q}_p)$ is maximal in the sense that $\Delta_f(0,\mathbb{Q}_p)$ is either the open or closed disc of radius $1/a$. In particular,  if
either $\max_{i\geq 2 }|a_i|^{1/(i-1)}$ is attained (as for polynomials) or $f$
diverges on $S_{1/a}(0,\mathbb{Q}_p)$, then $\Delta_f(0,\mathbb{Q}_p)=D_{1/a}(0,\mathbb{Q}_p)$.

\end{corollary}

\begin{proof}
If $\lambda$ is maximal, then $|1-\lambda ^m|=p^{-1}$ and $m=p-1$. Consequently,
$D_{\sigma_1}(0)=\overline{D}_{\sigma_1}(0) =D_{1/a}(0)$,
considered as discs in $\mathbb{Q}_p$.

\end{proof}
Moreover, a power series $f\in \mathcal{F}_{\lambda,a}\cap
\mathbb{Q}_p$ may diverge on $S_{1/a}(0)$. For example, the power
series $f(x)=\lambda x + \sum_{i=2}^{\infty}(a_2)^{i-1}x^i$
converges if and only if $|x|<1/|a_2|=1/a$.

Too see that $f$ may have a zero on the sphere $S_{1/a}(0)$,
consider the following example. Let $f(x)=\lambda x +a_2x^2$. Then
$a=|a_2|$. But $x=-\lambda /a_2\in\mathbb{Q}_p$ is a zero of $f$
located on the sphere $S_{1/a}(0)$ in $\mathbb{Q}_p$.


\begin{remark}  If $p$ is an odd prime
and $f(x)=\lambda x +O(x^2)$ is a power series over
$\mathbb{Z}_p$, with multiplier $\lambda $ such that $|1-\lambda
^m|=p^{-1}$, and $m=p-1$.  Then, the linearization disc in $\mathbb{Q}_p$
includes the open unit disc $D_{1}(0)$. This result was also
obtained in \cite[Proposition 2.2
]{Pettigrew/Roberts/Vivaldi:2001}.
\end{remark}

Let $K$ be an unramified field extension of $\mathbb{Q}_p$. Then
the value group $|K^*|$ contains only integer powers of $p$.
Hence, if $\lambda \in K$, then Theorem \ref{theorem case 1}
applies and we have the following corollary.

\begin{corollary}\label{corollary case 1}
Corollary \ref{corollary p-adic linearization disc} holds for any
unramified extension $K$ of $\mathbb{Q}_p$.
\end{corollary}

The case $p=2$, will be treated in case III below.

\subsection{Case II}

In this section we assume that $s\geq 1$ and
\begin{equation}\label{equation lambda case 2}
R(s)<|1-\lambda ^m|< R(s+1).
\end{equation}
In what follows $\sigma_2$ will be the real number defined by
\begin{equation}\label{definition of rho2}
\sigma_2:=a^{-1}p^{-\frac{1}{m(p-1)p^{s} }} |1-\lambda ^m|^{
\frac{1}{m}(1 + \frac{p-1}{p} s ) }.
\end{equation}

\begin{lemma}\label{lemma case 2}
Suppose $\lambda $ satisfies (\ref{equation lambda case 2}) for
some $s\geq 1$. Then,
\begin{equation}\label{prod ineq case 2}
\left( \prod_{n=1}^{k-1}|1-\lambda ^n| \right )^{-1}a^{k-1}\leq
p^{-\frac{1}{p-1}}\sigma_2^{-(k-1)},
\end{equation}
with equality if $(k-1)/mp^{s+1}$ is an integer power of $p$.
\end{lemma}

\begin{proof}
By lemma \ref{lemma distance 1-lambda m char 0,p}
\begin{equation}
\left |1 -\lambda ^n \right |= \left\{\begin{array}{ll}
1, & \textrm{if \quad $m\nmid n$,}\\
|n|p^s|1-\lambda ^m|^{p^s}, & \textrm{if \quad $m p^{s+1} \mid n$,}\\
|1-\lambda ^m|^{p^{\nu(n)}}, & \textrm{if \quad $m\mid n$ but
$mp^{s+1}\nmid n$}
\end{array}\right.
\end{equation}
in this case. Throughout this proof $N$ will be the integer
\begin{equation}\label{definition of N  case 1}
N=\left \lfloor \frac{k-1}{mp^{s+1}} \right \rfloor.
\end{equation}
Note that
\[
\prod_{mp^{s+1}\mid n}^{k-1}|n|= | mp^{s+1}\cdot 2mp^{s+1}\cdot
\dots |= | N!||mp^{s+1}|^{N},
\]
and since $p\nmid m$,
\begin{equation}\label{product |n| mp^s+1|n}
\prod_{mp^{s+1}\mid n}^{k-1}|n|= | N !|p^{-(s+1)N}.
\end{equation}
Moreover,
\begin{equation}\label{product p^s mp^s+1|n}
\prod_{mp^{s+1}\mid n}^{k-1}p^s|1-\lambda ^m|^{p^s}= p^{sN}
|1-\lambda ^m|^{p^sN},
\end{equation}
and
\begin{equation}\label{product p^nu p^s+1}
\prod_{\substack{ mp^{s+1}\nmid n \\ m\mid n} }^{k-1}|1-\lambda
^m|^{p^{\nu(n)}}= |1-\lambda ^m|^{ \sum_{j=0}^{s} \left (\left
\lfloor \frac{k-1}{mp^{j}}\right \rfloor - \left \lfloor
\frac{k-1}{mp^{j+1}}\right \rfloor \right )p^j }.
\end{equation}
Combining the three products (\ref{product |n| mp^s+1|n}),
(\ref{product p^s mp^s+1|n}), and (\ref{product p^nu p^s+1}) we
obtain
\begin{equation}\label{product 1 -lambda case 2}
\prod_{n=1}^{k-1}|1-\lambda ^n|= |N! |p^{-N} |1-\lambda
^m|^{\Sigma_1},
\end{equation}
where
\[
\Sigma_1 =p^sN+ \sum_{j=0}^s\left (\left \lfloor
\frac{k-1}{mp^{j}}\right \rfloor - \left \lfloor
\frac{k-1}{mp^{j+1}}\right \rfloor \right )p^j.
\]
Simplifying, we obtain
\[
\Sigma_1=\left \lfloor \frac{k-1}{m} \right \rfloor +
\sum_{j=1}^s\left\lfloor \frac{k-1}{mp^j}\right \rfloor
(p^j-p^{j-1}).
\]
Consequently,
\[
\Sigma_1 \leq \frac{k-1}{m}\left (1+s\frac{p-1}{p}\right ),
\]
with equality if $(k-1)/mp^{s+1}$ is an integer power of $p$. By
Lemma \ref{lemma order of p in factorial},
\[
\left |N!\right |p^{-N}\geq
p^{-\frac{N-1}{p-1}-N}=p^{-N\frac{p}{p-1} +\frac{1}{p-1}}\geq
p^{-\frac{k-1}{mp^s(p-1)} +\frac{1}{p-1}},
\]
where each inequality become an equality if $(k-1)/mp^{s+1}$ is an
integer power of $p$. Applying these estimates to the identity
(\ref{product 1 -lambda case 2}) we obtain the inequality
(\ref{prod ineq case 2}) as required.
\end{proof}


By similar arguments as those applied in the proof of Theorem
\ref{theorem case 1}, we obtain the following result.

\begin{theorem}\label{theorem case 2}
Let $f\in \mathcal{F}_{\lambda,a}$ and suppose $\lambda $
satisfies
\[
R(s)<|1-\lambda |<R(s+1),
\]
for some integer $s\geq 1$. Then, $\Delta_f\supseteq
D_{\sigma_2}(0)$. Moreover, if the conjugacy function $g$
converges on the closed disc $\overline{D}_{\sigma_2}(0)$, then
$\Delta_f\supseteq \overline{D}_{\sigma_2}(0)$.
\end{theorem}


\subsection{Case III}

In this section it will be assumed that $s\geq 1$ and
\begin{equation}\label{equation lambda case 3}
|1-\lambda ^m|= R(s).
\end{equation}
In what follows $\sigma_3$ will be the real number defined by
\begin{equation}\label{definition of rho3}
\sigma_3:=\sigma_2\cdot\left ( \frac{|\alpha-\lambda
^m|}{|1-\lambda ^m |} \right )^{1/mp^s},
\end{equation}
where $\alpha\in \Gamma_r$ is chosen such that
$|\alpha-\lambda^m|\leq |\gamma -\lambda^m|$, for all
$\gamma\in\Gamma$.

\begin{lemma}\label{lemma case 3}
Suppose $\lambda $ satisfies (\ref{equation lambda case 3}) for
some $s\geq 1$. Then
\begin{equation}\label{prod ineq case 3}
\left( \prod_{n=1}^{k-1}|1-\lambda ^n| \right )^{-1}a^{k-1}\leq
p^{-\frac{1}{p-1}}\sigma_3^{-(k-1)},
\end{equation}
with equality if $(k-1)/mp^{s}$ is an integer power of $p$.
\end{lemma}

\begin{proof}
By Lemma \ref{lemma distance 1-lambda m char 0,p}
\begin{equation}
\left |1 -\lambda ^n \right |= \left\{\begin{array}{ll}
1, & \textrm{if \quad $m\nmid n$,}\\
|n|p^s|1-\lambda ^m|^{p^s-1}|\alpha -\lambda ^m|, & \textrm{if \quad $m p^{s} \mid n$,}\\
|1-\lambda ^m|^{p^{\nu(n)}}, & \textrm{if \quad $m\mid n$ but
$mp^{s}\nmid n$}
\end{array}\right.
\end{equation}
in this case. Throughout this proof we
let 
$M$ be the integer
\[
M=\left \lfloor \frac{k-1}{mp^s} \right \rfloor.
\]
Note that
\begin{equation}\label{product |n| mp^s|n}
\prod_{mp^{s}\mid n}^{k-1}|n|= |M!|p^{-sM},
\end{equation}
\begin{equation}\label{product p^s mp^s|n}
\prod_{mp^{s}\mid n}^{k-1}p^s|1-\lambda ^m|^{p^s}=
p^{sM}|1-\lambda ^m|^{(p^s-1)M},
\end{equation}
\begin{equation}\label{product alpha mp^s|n}
\prod_{mp^{s}\mid n}^{k-1}|\alpha -\lambda ^m|= |\alpha - \lambda
^m|^{M},
\end{equation}
and
\begin{equation}\label{product p^nu p^s}
\prod_{\substack{ mp^{s}\nmid n \\ m\mid n} }^{k-1}|1-\lambda
^m|^{p^{\nu(n)}}= |1-\lambda ^m|^{ \sum_{j=0}^{s-1} \left (\left
\lfloor \frac{k-1}{mp^{j}}\right \rfloor - \left \lfloor
\frac{k-1}{mp^{j+1}}\right \rfloor \right )p^j }.
\end{equation}
Combining the three products (\ref{product |n| mp^s|n}),
(\ref{product p^s mp^s|n}), (\ref{product alpha mp^s|n}), and
(\ref{product p^s mp^s|n}) we obtain
\begin{equation}\label{product 1 -lambda case 3}
\prod_{n=1}^{k-1}|1-\lambda ^n|=|M!||1-\lambda ^m|^{\Sigma_2}\left
(  \frac{|\alpha -\lambda ^m|}{|1-\lambda ^m|} \right )^{M},
\end{equation}
where
\[
\Sigma_2=p^sM+\sum_{j=0}^{s-1} \left (\left \lfloor
\frac{k-1}{mp^{j}}\right \rfloor - \left \lfloor
\frac{k-1}{mp^{j+1}}\right \rfloor \right )p^j.
\]
Simplifying, we obtain
\[
\Sigma_2=\Sigma_1=\left \lfloor \frac{k-1}{m} \right \rfloor +
\sum_{j=1}^s\left\lfloor \frac{k-1}{mp^j}\right \rfloor
(p^j-p^{j-1}).
\]
As in the proof of Lemma \ref{lemma case 2} we have
\[
\Sigma_1 \leq \frac{k-1}{m}\left (1+s\frac{p-1}{p}\right ),
\]
with equality if $(k-1)/mp^{s}$ is an integer power of $p$. By
Lemma \ref{lemma order of p in factorial},
\[
\left |M!\right |\geq p^{-\frac{M-1}{p-1}}\geq
p^{-\frac{k-1}{mp^s(p-1)} +\frac{1}{p-1}},
\]
where each inequality become an equality if $(k-1)/mp^{s}$ is an
integer power of $p$. Furthermore, by definition
\[
|\alpha -\lambda^ m|\leq |1-\lambda ^m|=R(s).
\]
Applying these estimates to the identity (\ref{product 1 -lambda
case 3}) we obtain the inequality (\ref{prod ineq case 3}) as
required.

\end{proof}

By similar arguments as those applied in the proof of Theorem
\ref{theorem case 1}, we obtain the following result.

\begin{theorem}\label{theorem case 3}
Let $f\in \mathcal{F}_{\lambda,a}$ and suppose
\[
|1-\lambda ^m|=R(s),
\]
for some integer $s\geq 1$. Then, $\Delta_f\supseteq
D_{\sigma_3}(0)$. Moreover, if the conjugacy function $g$
converges on the closed disc $\overline{D}_{\sigma_3}(0)$, then
$\Delta_f\supseteq \overline{D}_{\sigma_3}(0)$.
\end{theorem}

\begin{corollary}\label{corollary case 3}
Let $K$ be an unramified extension of $\mathbb{Q}_2$.
Let $f$ be of the form
\[
f(x)=\lambda (x-x_0)+\sum a_i(x-x_0)^i\in K[[x-x_0]] ,  a=\sup |a_i|^{1/(i-1)}.
\] 
Then, the following two statements hold: 
\begin{enumerate}

\item If $|1-\lambda |<1/2$, then the linearization disc $\Delta_f(x_0,K)$ contains the open disc of radius
         $\sigma_1= |1-\lambda |/2a $ about $x_0$.

\item If $|1-\lambda |=1/2$, then the linearization disc the linearization disc $\Delta_f(x_0,K)$ contains the open disc of radius   $\sigma_3=\sqrt{ |1+\lambda|}/a$ about $x_0$.

\end{enumerate}
\end{corollary}

\begin{proof}
As $p=2$ we must have $m=1$ and since $K$ is unramifed we must have $s\leq1$. The first statement is then a direct consequence of Theorem \ref{theorem case 1}. As to the second statement, $|1-\lambda| =1/2$, Theorem \ref{theorem case 3} applies with $m=1$, $s=1$ and $\alpha =-1$. Hence, 
$\sigma_3=\sqrt{ |1+\lambda|}/2^{3/2}a$, and since $K$ is unramified we may as well exclude the factor 
$2^{3/2}$.
\end{proof}

\begin{remark}
Note that $\sqrt{ |1+\lambda|}\leq 1/2$, so even if $\lambda $ is maximal (as in the second  statement  of Corollary \ref{corollary case 3} above), it seems that the radius of the linearization disc in $\mathbb{Q}_2$ may not be the maximal, $1/a$ as obtained for $\mathbb{Q}_p$ with $p$ odd in Corollary 
\ref{corollary p-adic linearization disc}.
   
\end{remark}

\subsection{Statement of the general estimate}

Suppose $\lambda ^m$ belongs to the annulus
\[
\{z\in\mathbb{C}_p:R(s)<|1-z|<R(s+1)\}.
\]
Then, in view of Lemma \ref{lemma closest root of unity},
$\alpha=1$ is the closest root of unity to $\lambda ^m$. It
follows that $\sigma_3=\sigma_2$ in this case. Consequently, the
estimate $\sigma_3$ of the radius of the linearization disc, holds for all $\lambda $
such that $R(s)\leq |1-\lambda ^m|<R(s+1)$ for some $s\geq 1$.
Furthermore, if we put $s=0$ and $\alpha =1$, then
$\sigma_3=\sigma_1$.

Hence, $\sigma_3$ may serve as a general bound if we include the
case $s=0$. Recall that, by definition, $R(0)=0$.
Our estimates can thus be summarized according to the following
theorem.
\begin{theorem}\label{theorem general estimate}
Let $f\in \mathcal{F}_{\lambda,a}$.
Suppose $\lambda $ is not a root of unity, and
\[
R(s)\leq |1-\lambda ^m|<R(s+1),
\]
for some $s\geq 0$. Then, the linearization disc $\Delta_f\supseteq
D_{\sigma}(0)$ where

\begin{equation}\label{definition sigma}
\sigma=\sigma(\lambda,a):=a^{-1}R(s+1)^{\frac{1}{m}} |1-\lambda
^m|^{ \frac{1}{m}(1 + \frac{p-1}{p} s ) }\left (
\frac{|\alpha-\lambda ^m|}{|1-\lambda ^m |} \right )^{1/mp^s}.
\end{equation}
Moreover, if the conjugacy function converges on the closed disc
$\overline{D}_{\sigma}(0)$, then $\Delta_f\supseteq
\overline{D}_{\sigma}(0)$.
\end{theorem}

\subsection{Asymptotic behavior of the estimate of the radius of the linearization disc}\label{section asympt behav siegel rad}

In this section we consider the following question. What happens
to the estimate
\[
\sigma=a^{-1}p^{-\frac{1}{m(p-1)p^{s} }} |1-\lambda ^m|^{
\frac{1}{m}(1 + \frac{p-1}{p} s ) }\left ( \frac{|\alpha-\lambda
^m|}{|1-\lambda ^m |} \right )^{1/mp^s},
\]
of the radius of the linearization disc, as $m$ or $s$ goes to infinity? We will show
that in each of these two cases $\sigma$ approach $1/a$. As stated
in Lemma \ref{lemma upper bound Siegel and isometry}, for $f\in
\mathcal{F}_{\lambda,a}$, the value $1/a$ is the maximal radius of
a linearization disc. On the other hand, if $s$ and $m$ are fixed, then
$\sigma \to 0$ as $|\alpha -\lambda ^m|\to 0$.


The result is based on the following Lemma.
\begin{lemma}\label{lemma bound for sigma}
For all $s\geq 0$
\[
\sigma\geq a^{-1}p^{-\frac{1+s(p-1)}{m(p-1)p^s}}|1-\lambda
^m|^{\frac{1}{m}}\left ( \frac{|\alpha-\lambda ^m|}{|1-\lambda ^m
|} \right )^{1/mp^s}.
\]
In particular,
\[
\sigma\geq a^{-1}p^{-\frac{1}{m(p-1)}}|\alpha -\lambda
^m|^{\frac{1}{m}}.
\]
\end{lemma}
\begin{proof}
Recall that for $s=0$, we have that $|\alpha -\lambda
^m|=|1-\lambda|$. This completes the case $s=0$.

For $s\geq 1$, we have $|1-\lambda ^m|\geq p^{-1/(p^{s-1}(p-1))}$,
and  by the definition of $\sigma$
\[
\sigma\geq a^{-1}p^{-\frac{1+s(p-1)}{m(p-1)p^s}}|1-\lambda
^m|^{\frac{1}{m}}\left ( \frac{|\alpha-\lambda ^m|}{|1-\lambda ^m
|} \right )^{1/mp^s}.
\]
Recall from Remark \ref{remark lemma distance 1-lambda m char 0,p}
that we always have $|\alpha -\lambda ^m |\leq |1-\lambda ^m|$.
Hence, we also obtain the following bound for $\sigma$
\[
\sigma\geq  a^{-1}p^{-\frac{1}{m(p-1)}}|\alpha -\lambda
^m|^{\frac{1}{m}},
\]
as required.
\end{proof}
Note that if $|1-\lambda ^m|\neq R(s)$ for all $s\geq 0$, then
$\alpha =1$. Hence, by increasing $s$, we push $\lambda ^m$
further away from the closest root of unity, and make $\sigma $
closer to its maximum value $1/a$. Loosely speaking, according to
the the Lemma, the farther $\lambda ^m$ is from the `closest' root
of unity, the closer $\sigma$ is to its maximum value $1/a$. On
the other hand, if $s$ and $m$ are fixed, then $\sigma \to 0$ as
$|\alpha -\lambda ^m|\to 0$. In particular, we have the following
theorem.

\begin{theorem}\label{theorem asymptotic behavior sigma}
Let $|\alpha -\lambda ^m|$ be fixed. Then, the estimate $\sigma$
of the radius of the linearization disc goes to $1/a$ as $m$ or $s$ goes to infinity.
If $s$ and $m$ are fixed, then $\sigma \to 0$ as $|\alpha -\lambda
^m|\to 0$.
\end{theorem}

\subsection{Maximal linearization discs in finite extensions of  $\mathbb{Q}_p$}
Let $K$ be a finite extension of $\mathbb{Q}_p$, and let  $f\in \mathcal{F}_{\lambda,a}\cap K[[x]]$. The disc  
\[
\Delta_f(0,K)=\Delta_f\cap K
\]
will be refered to as the corresponding linearization disc in $K$. We say
that   $\Delta_f(0,K)$ is \emph{maximal} if it contains the open disc $D_{1/a}(0,K)=D_{1/a}(0)\cap K$.

\begin{theorem}\label{theorem maximal linearization disc in K }
Let $K$ be a finite extension of $\mathbb{Q}_p$, with ramification index $e$. Let  $f\in \mathcal{F}_{\lambda,a}\cap K[[x]]$ and let $s$ be the integer for which
$R(s)<|1-\lambda^m|\leq R(s+1)$.
Let  $\epsilon$ be the integer satisfying $\nu(1-\lambda^m )=\epsilon/e$. Suppose that
\begin{equation}\label{general condition maximal linearization disc}
s  < \left (\frac{m}{\epsilon }-2\right)\frac{p}{p-1} - \nu\left ( \frac{\alpha - \lambda ^{m}}{1-\lambda^{m}} 
\right ).
\end{equation}
Then, the linearization disc $\Delta_f(x_0,K)$ is maximal. In particular,  if
either $\max_{i\geq 2 }|a_i|^{1/(i-1)}$ is attained (as for polynomials) or $f$
diverges on $S_{1/a}(x_0,K)$, then $\Delta_f(x_0,K)=D_{1/a}(x_0,K)$.
\end{theorem}

\begin{proof}
We first consider the case $R(s)<|1-\lambda^m|< R(s+1)$. Consequently,
$\epsilon p^{s-1}(p-1)<e<\epsilon p^{s}(p-1)$. In view of (\ref{definition sigma})
with $\alpha = 1$,
\[
\nu(a\sigma)=\frac{1}{m(p-1)p^s}+ \frac{\epsilon(1+ s(p-1)/p)}{em}=
\frac{e+\epsilon p^s(p-1)(1+s(p-1)/p)}{em(p-1)p^s},
\]
and since $e<\epsilon p^s(p-1)$ we have
\[
\nu (a\sigma)<\frac{1}{e}\cdot \frac{\epsilon (2+s(p-1)/p)}{m}.
\]
It follows that 
\[
\nu(a\sigma)<1/e
\]
if $s<p(m/\epsilon -2)/(p-1)$ as required.

We now consider the case $|1-\lambda^m|= R(s+1)$. Hence,
$e=\epsilon p^{s}(p-1)$. In view of (\ref{definition sigma}),
\[
\nu(a\sigma)=\frac{1}{m(p-1)p^{s+1}}+ \frac{\epsilon(1+ (s+1)(p-1)/p)}{em}+\frac{\nu}{mp^{s+1}}.
\]
Hence,
\[
\nu(a\sigma)=
\frac{e+\epsilon p^{s+1}(p-1)(1+(s+1)(p-1)/p)+e\nu (p-1)}{em(p-1)p^{s+1}},
\]
and since $e=\epsilon p^s(p-1)$ we have
\[
\nu (a\sigma)=\frac{1}{e}\cdot \frac{\epsilon p(1/p+1+(s+1)(p-1)/p)+(p-1)\nu }{mp},
\]
or equivalently,
\[
\nu (a\sigma)=\frac{1}{e}\cdot \frac{\epsilon (2+s(p-1)/p)+\nu (p-1)/p }{m}.
\]
It follows that 
\[
\nu(a\sigma)<1/e
\]
if $s<p(m/\epsilon -2)/(p-1)-\nu$ as required.
\end{proof}

One might ask whether the condition (\ref{general condition maximal linearization disc}) is really 
necessary  or just a consequence of lack of precision in our estimates of the radius of the linearization disc.
However, in section \ref{section quadratic case} (Corollary \ref{corollary quadratic power series}) we show that there are examples where  the condition (\ref{general condition maximal linearization disc}) is not satisfied and the corresponding linearization disc is strictly contained in the disc $D_{1/a}(0)$.   

\begin{corollary}
Let $K$ be a finite extension of $\mathbb{Q}_p$ of degree $n$, with ramification index $e$, and residue field  $k$ of  degree $[k:\mathbb{F}_p]=n/e$. Let $\lambda $ be maximal,  $f\in \mathcal{F}_{\lambda,a}\cap K[[x]]$, and let $s$ be the integer for which $R(s)<|1-\lambda^{p^{n/e}-1}|\leq R(s+1)$.
Suppose that
\begin{equation}\label{condition maximal linearization disc lambda maximal}
s  < (p^{n/e} - 3)\frac{p}{p-1} - \nu\left ( \frac{\alpha - \lambda ^{p^{n/e}-1}}{1-\lambda^{p^{n/e}-1}} 
\right ).
\end{equation}Then, the linearization disc $\Delta_f(0,K)$ is maximal. In particular,  if
either the maximum $\max_{i\geq 2 }|a_i|^{1/(i-1)}$ is attained (as for polynomials) or $f$
diverges on $S_{1/a}(0,K)$, then $\Delta_f(0,K)=D_{1/a}(0,K)$.
\end{corollary}
\begin{proof}
Recall that since $\lambda $ is maximal we have $m=p^{n/e}-1$ and moreover $|1-\lambda^{p^{n/e}-1}|=p^{-1/e}$.
The corollary then  follows from theorem \ref{theorem maximal linearization disc in K } with  $\epsilon =1$ and $m=p^{n/e}-1$.    
\end{proof}

\begin{remark}

For example, if $K$ is an unramified extension, then $e=1$ and $s=0$. Furthermore, if $\lambda$ is
maximal and $p$ is odd, then $\alpha=1$. Hence, if $p>3$, the condition (\ref{condition maximal linearization disc lambda maximal})   holds and the linearization disc is maximal. On the other hand, if $p=3$, then  condition (\ref{condition maximal linearization disc lambda maximal}) is not satisfied. However, by Corollary \ref{corollary case 1} we know that the linearization disc is maximal also for $p=3$.
This shows that the condition (\ref{condition maximal linearization disc lambda maximal}) is not necessary in this case.
\end{remark}

\section{The quadratic case}\label{section quadratic case}

To see the results at work we provide
examples, where we can find the exact size of the linearization disc. We
begin with quadratic polynomials, and then show how we can extend
this result to power series containing a `sufficiently large'
quadratic term. We also give sufficient conditions, on the
multiplier $\lambda$, that there is a fixed point on the
`boundary' of a linearization disc for quadratic polynomials.


Given a quadratic polynomial $f$ of the form $f(x)=\lambda
x+a_2x^2\in\mathbb{C}_p[x]$, the radius of the corresponding
linearization disc can be estimated by $\sigma$, defined by
$(\ref{definition sigma})$. If $\lambda $ is located inside the
annulus $\{z: p^{-1}<|1-z|<1\}$, we can actually find the exact
size of the linearization disc.

\begin{theorem}\label{theorem quadratic polynomials}
Let $p$ be an odd prime and let $\lambda\in\mathbb{C}_p$, not a
root of unity, belong to the annulus $\{z:p^{-1}<|1-z|<1\}$. Let
$f$ be a quadratic polynomial of the form $f(x)=\lambda
x+a_2x^2\in\mathbb{C}_p[x]$. Then, the coefficients of the
conjugacy $g$ satisfy
\begin{equation}\label{bk quadratic simplified}
|b_{k}|=\frac{|a_2|^{k-1} |1-\lambda|^{ \lfloor \frac{k-1}{p}
\rfloor }} {\prod_{n=1}^{k-1}|1-\lambda^n|}.
\end{equation}
Moreover, the linearization disc about the origin,
$\Delta_f=D_{\tau}(0)$, where the radius $\tau=|1-\lambda|^{-1/p}\sigma$.
\end{theorem}

\begin{proof}
We first prove that the coefficients of the conjugacy, defined by
the recursive formula (\ref{bk-equation}), satisfy the identity
(\ref{bk quadratic simplified}).

Recall that one consequence of ultrametricity is that for any
$x,y\in \mathbb{K}$ with $|x|\neq |y|$, the inequality (\ref{sti})
becomes an equality. In other words, if $x,y\in \mathbb{K}$ with
$|x|<|y|$, then $|x+y|=|y|$. The idea of the proof is to find a
dominating term in the right hand side of (\ref{bk-equation})
which is strictly greater than all the others. Then, the absolute
value of the coefficient $b_k$, of the conjugacy $g$, is equal to
the absolute value of the dominating term. The proof is similar to
that performed in \cite[p 760--761]{Lindahl:2004}, for fields of
prime characteristic.

Since the term $l!/(\alpha_1!\alpha_2!)$ is always an integer
\[
|\frac{l!}{\alpha_1!\alpha_2!}|\leq1,
\]
with equality if and only if $l!/(\alpha_1!\alpha_2!)$ is not
divisible by $p$. We will show that most of the time the
$b_{k-1}$-term is the greatest. In fact,
\begin{equation}\label{equation factorial k-1}
\frac{(k-1)!}{\alpha_1!\alpha_2!}=k-1.
\end{equation}
As $k$ runs from $1, \dots, p$, the number $k-1$ will never be
divisible by $p$. Recall that we assume in this proof that
$|1-\lambda |<1$ (so that $m=1$). Hence, $|1-\lambda ^n|<1$, for
all integers $n\geq 1$. Therefore, the $b_{k-1}$-term will be
strictly greater than all the other terms in the right hand side
of (\ref{bk-equation}), and thus by the ultrametric triangle
inequality (\ref{sti}), we have
\begin{equation}\label{bk1}
|b_k|=\frac{|b_{k-1}||k-1||a_2|}{|1-\lambda^{k-1}|}=\frac{|a_2|^{k-1}}{|1-\lambda^{k-1}||1-\lambda^{k-2}|\dots
|1- \lambda|}.
\end{equation}
But if $k=p+1$, then $|k-1|= p^{-1}$ so that for $l=k-1$ in
(\ref{bk-equation}), we obtain
\begin{equation}\label{bk-1 term k-1 divisible by p}
|b_{k-1}|\left|\frac{(k-1)!}{\alpha_1!\alpha_2!}\right||a_2|=\frac{p^{-1}|a_2|^{p}}{|1-\lambda^{p-1}||1-\lambda^{p-2}|\dots
|1-\lambda|}
\end{equation}
Then, the $b_{k-2}$-term will dominate. In fact,
\begin{equation}\label{faculty k-2}
\left |\frac{(k-2)!}{\alpha_1!\alpha_2!}\right
|=\left|\frac{(k-2)(k-3)}{2}\right|=1, \text{ if } p|k-1, \text{ }
p>2.
\end{equation}
As a consequence, $l=k-2$ gives
\begin{equation}\label{bk-2 term k-1 divisible by p}
|b_{p-1}|\left |\frac{(k-2)!}{\alpha_1!\alpha_2!}\right
||a_2|^2=\frac{|a_2|^{p}}{|\lambda^{p-2}-1||\lambda^{p-3}-1|\dots
|\lambda-1|}.
\end{equation}
Note that, since $m=1$ in our case, we have $|1-\lambda
^{p-1}|=|1-\lambda |$. Moreover, by assumption, $|1-\lambda
|>p^{-1}$. Hence, the $b_{k-2}$-term (\ref{bk-2 term k-1 divisible
by p}) is strictly greater than the $b_{k-1}$-term (\ref{bk-1 term
k-1 divisible by p}), and all $b_l$-terms for which $l<k-2$.
Consequently,
\begin{equation}\label{bk2}
|b_{p+1}|=\frac{|b_{p-1}||a_2|^2}{|\lambda^{p}-1|}=\frac{|a_2|^{p}}{|\lambda^{p}-1||\lambda^{p-2}-1||\lambda^{p-3}-1|\dots
|\lambda-1|}.
\end{equation}
Note the lack of the factor $|\lambda^{p-1}-1|$. Now, since
according to Lemma \ref{lemma distance 1-lambda m char 0,p}
\[
|1-\lambda^p|<|1-\lambda|=|1-\lambda^{p-1}|,
\]
we have that
\[
|1-\lambda^{p}|\prod_{j=1}^{p-2}|1-\lambda^{j}|<|1-\lambda^{p-1}|\prod_{j=1}^{p-2}|1-\lambda^{j}|.
\]
Therefore we have for $k=p+2$ that the $b_{k-1}$-term is again
strictly greater than all the others in the right hand side of
(\ref{bk-equation}) so that
\begin{equation*}
|b_{p+2}|=\frac{|b_{p+1}||a_2|}{|\lambda^{p+1}-1|}=\frac{|a_2|^{p+1}}{|1-\lambda^{p+1}||1-\lambda^p||1-\lambda^{p-2}||1-\lambda^{p-3}|\dots
|1-\lambda |}.
\end{equation*}
The $b_{k-1}$-term will dominate until $p$ divides $k-1$ again,
i.e.\@ for $k=2p+1$ (which means that we ``loose'' the factor
$|\lambda^{2p-1}-1|$). Repeated application of these arguments
yields that
\begin{equation}\label{bk quadratic}
|b_{k}|=\frac{|a_2|^{k-1}\prod_{i\cdot p\leq
k-1}|1-\lambda^{ip-1}|} {\prod_{n=1}^{k-1}|1-\lambda^n|}.
\end{equation}
Note that $|1-\lambda^{ip-1}|=|1-\lambda |$, since $m=1$ in this
case. Hence, we obtain (\ref{bk quadratic simplified}) as
required.

It remains to prove that that the corresponding linearization disc is the
open disc $D_{\tau}(0)$. Recall that the estimates for the $b_k$:s
in the previous sections where based on the estimate (\ref{b_k
estimate by prod 1-lambda n}). Moreover,
\[
\left (|1-\lambda|^{ \lfloor \frac{k-1}{p} \rfloor }\right
)^{1/k}\to |1-\lambda |^{\frac{1}{p}}, \quad k\to \infty.
\]
This suggests that $g$ converges on an open disc of radius
\[
\tau =|1-\lambda |^{-\frac{1}{p}}\sigma.
\]
In fact, $g$ diverges on the sphere of radius $\tau$; let $I\geq
s+1$ be an integer, then by Lemma \ref{lemma case 1}, \ref{lemma
case 2}, \ref{lemma case 3} and (\ref{bk quadratic simplified}) we
have
\begin{equation}\label{bk tau k quadratic}
|b_{p^I+1}|\tau ^{p^I+1}=p^{-\frac{1}{p-1}}
|1-\lambda|^{-\frac{1}{p}}\sigma=p^{-\frac{1}{p-1}}\tau ,
\end{equation}
which does not approach zero as $I$ goes to infinity. Furthermore,
in a similar way, applying Lemma \ref{lemma case 1}, \ref{lemma
case 2}, and \ref{lemma case 3} to the identity (\ref{bk quadratic
simplified}) we obtain
\[
|b_{k}|\tau ^{k}\leq p^{-\frac{1}{p-1}}\tau.
\]
Consequently, $g$ is one-to-one on $D_{\tau}(0)$. It follows that
the linearization disc of the quadratic polynomial $f$ is the disc
$\Delta_f=D_{\tau}(0)$.

\end{proof}

The theorem implies, in particular, that $f$ can have no periodic
points (except the fixed point at the origin) in the disc
$D_{\tau}(0)$. However, there may be periodic points on the
boundary. We will give sufficient conditions that there is a fixed
point on the boundary $S_{\tau}(0)$. Note that $f$ has a fixed
point $\hat{x}=(1-\lambda)/a_2$. Solving the equation
\begin{equation}\label{equation fixed point}
\tau=|1-\lambda |/|a_2|,
\end{equation}
yields that $\hat{x}$  is located on $S_{\tau}(0)$ if $s=1$ and
$\alpha =1$ and $|1-\lambda |=p^{-1/2(p-1)}$, or if $s\geq 2$ and
$|1-\lambda |=p^{-t(s)}$, where
$t(s)=[(s-1)p^{s-1}(p-1)+p^{s-1}-1]/p^{s-1}(p-1)$. Furthermore,
$\hat{x}$ cannot be located on $S_{\tau}(0)$ if $s=0$; the only
solution to (\ref{equation fixed point}) for $s=0$ is $|1-\lambda
|=p^{-p/(p-1)}$, and hence, $\lambda $ does not belong to the
annulus $\{z:p^{-1}<|1-z|<1\}$.

As in \cite[Corollary 2.1]{Lindahl:2004} the previous result on
quadratic polynomials works also for power series containing a
dominating quadratic term in the following sense.

\begin{theorem}\label{theorem quadratic power series}
Let $p$ be an odd prime and let $\lambda\in\mathbb{C}_p$, not a
root of unity, belong to the annulus $\{z:p^{-1}<|1-z|<1\}$. Let
$f\in \mathbb{C}_p[[x]]$ be a power series of the form
\[
f(x)=\lambda x + a_2x^2 +\sum_{i\geq 3}a_ix^i,
\]
where
\begin{equation}\label{condition quadratic power series}
|1-\lambda |^{1/p}|a_2|>1, \quad |1-\lambda |^{1/p}|a_2|>|a_i|,
\quad i\geq 3.
\end{equation}
Then, the coefficients of the conjugacy $g$ satisfy (\ref{bk
quadratic simplified}). Moreover, the linearization disc about the
origin, $\Delta_f=D_{\tau}(0)$, where
$\tau=|1-\lambda|^{-1/p}\sigma$.
\end{theorem}

\begin{proof}
By the condition (\ref{condition quadratic power series}), the
same terms as in the proof of Theorem \ref{theorem quadratic
polynomials} will be stricly larger than all the others in
(\ref{bk-equation}). The reason for the factor
$|1-\lambda|^{1/p}$, is the lack of the factor $|1-\lambda
^{p-1}|$ in (\ref{bk2}).
\end{proof}

\begin{corollary}\label{corollary quadratic power series}
Let  $f$ be a power series satisfying the condions of Theorem \ref{theorem quadratic power series}, 
then the radius of the linearization disc 
\[
\tau <p^{-1/(p-1)p^s}a^{-1}.
\]
In particular, the linearization disc cannot contain the  disc $D_{1/a}(0)\cap K$ for any algebraic 
extension $K$ of $\mathbb{Q}_p$. 
\end{corollary}

\section{Minimality and ergodicity}

\subsection{Minimality and conjugation in non-Archimedean fields}

In this section we consider power series defined over an arbitrary
complete non-Archimedean field $K$, rather than just over
$\mathbb{C}_p$. The notion of transitivity and minimality on
subsets of $K$ are defined as follows. Let $X$ be a subset of $K$
and let $f\in K[[x]]$ be a power series which converges on $X$.
Suppose that $X$ is invariant under $f$, i.e.\@ $f(X)\subseteq X$.
The map $f:X\to X$ is said to be \emph{transitive} if there is an
element $x\in X$, such that its forward orbit $\{f^{\circ
n}(x)\}_{n=0}^{\infty}$ is dense in $X$. We say that $f:X\to X$ is
\emph{minimal} on $X$ if for every $x\in X$, its forward orbit
$\{f^{\circ n}(x)\}_{n=0}^{\infty}$ is dense in $X$.

We will look for dense orbits near indifferent non-resonant fixed points of $f$.
By Lemma \ref{lemma linearization disc
isometry}, the dynamics on  a linearization disc $\Delta _f(x_0,K)$ is located on 
invariant spheres, about the fixed point $x_0$.
As in previous sections we will assume (without loss of
generality) that $f$ has an indifferent fixed point at the origin.
Given $\lambda\in K$, let $T_{\lambda}: K \to K$ be the multiplication map, $x
\mapsto \lambda x$. We will prove that transitivity of
$f(x)= \lambda x +O(x^2)\in K[[x]]$, on some subset $X$,
of the corresponding linearization disc $\Delta_f$ about the origin, is
equivalent to transitivity of $T_{\lambda}$ on $g(X)$. Moreover,
transitivity and minimality are equivalent if $X$ is compact. 


\begin{theorem}[Transitivity is preserved under analytic conjugation]\label{theorem-minimalitypreserved}
Let $f(x)=\lambda x +O(x^2)\in K[[x]]$ be analytically
conjugate to $T_{\lambda}$, on a linearization disc $\Delta_f$ about the
origin, via a conjugacy function $g$, such that $g(0)=0$ and $g'(0)=1$.
Suppose that the subset $X\subseteq \Delta_f$ is invariant under
$f$. Then, the following statements hold:
\begin{enumerate}[1)]

 \item $f$ is transitive on $X$ if and only if $T_{\lambda}$ is transitive on $g(X)$.
 \item If $X$ is compact and $f$ is transitive on
$X$, then $f$ is minimal on $X$. Moreover, $f(X)=X$ and
$g(X)=T_{\lambda}(g(X))$.
\end{enumerate}
\end{theorem}
\begin{proof}
Let $f$ be analytically conjugate to $T_{\lambda}$, with conjugacy
function $g$.
Suppose that $X\subseteq \Delta_f$ is invariant under $f$. By the conjugacy relation we must have $g(f(X))=T_{\lambda}(g(X))$. It follows that
$T_{\lambda}(g(X))\subseteq g(X)$. In other words, $g(X)$ is
invariant under $T_{\lambda}$.

Now, suppose that $T_{\lambda}$ is transitive on $g(X)$. Then,
there is an $x\in X$ such that the orbit $\{T_{\lambda}^{\circ
n}(x)\}$ is dense in $g(X)$. Recall that by
Lemma \ref{lemma linearization disc isometry}, $g:X\to g(X)$ is bijective
and isometric. Hence, given $\epsilon
>0$ and $y\in X$, there is an integer $n\geq 1$ such that in
view of the conjugacy relation
  \begin{eqnarray*}
    \epsilon &>& |T_{\lambda}^{\circ{n}}\circ g(x)-g(y)|=|g^{-1}\circ T_{\lambda}^{\circ{n}}\circ g(x)-g^{-1}\circ
    g(y)|=|f^{\circ n}(x)-y|.
  \end{eqnarray*}
It follows that the orbit $\{f^{\circ n}(x)\}$ is dense in $X$.
Accordingly, $f$ is transitive on $X$. Likewise, transitivity of
$f$ implies transitivity of $T_{\lambda}$.

Now we consider the second statement of the theorem. Suppose that the
subset $X\subset \Delta _f$ is compact, and that $f:X\to X$ is
transitive so that $T_{\lambda}$ is transitive on $g(X)$. In view
of the conjugacy relation we have $g(f^{\circ
n}(x))=T_{\lambda}^{\circ n}(g(x))$. Hence, transitivity of
$T_{\lambda}$ implies that $g(X)$ is dense in $T_{\lambda}(g(X))$.
Continuity of $g$ and $T_{\lambda}$, and compactness of $X$, gives
$g(X)=T_{\lambda}(g(X))$. Consequently, $f(X)=g^{-1}\circ
T_{\lambda} \circ g(X)=X$. Hence, $f:X\to X$ is not only
one-to-one and isometric but also surjective.

It is well known that a transitive bijective isometry is minimal,
see e.g.\@ \cite{Walters:1982}. Accordingly, minimality and
transitivity are equivalent on compact subsets of non-Archimedean
linearization discs. This completes the proof.
\end{proof}

\begin{remark}\label{remark f isometry X subset sphere}
By Lemma \ref{lemma linearization disc isometry}, $f:\Delta_f\to
\Delta_f$ is also an isometry. Therefore, if $f$ is minimal on
some subset $X\subseteq \Delta_f$ and $X\neq \{0\}$, then
$X\subseteq S$ for some sphere $S\subset \Delta_f$.
\end{remark}

\subsection{Minimality in $\mathbb{Q}_p$}


It follows from the previous section that transitivity, and hence
minimality, on spheres about an indifferent fixed point can be
characterized in terms of the multiplier map. By Lemma \ref{lemma transitivity multiplier},
the multiplier map $T_{\lambda}$ is transitive on a sphere about the origin in 
$\mathbb{Q}_p$ if and only if $\lambda$ is maximal and $p$ is odd.
Moreover, in view of Theorem \ref{theorem-minimalitypreserved}, minimality and transitivity 
of $T_{\lambda}$ coincide, as proven earlier by various authors
\cite{Bryk/Silva:2003,Coelho/Parry:2001,Gundlach/Khrennikov/Lindahl:2001:a,Oselies/Zieschang:1975} in the $p$-adic setting.

Recall that, if $\lambda$ is maximal, then  the corresponding
linearization disc is maximal as stated in Corollary \ref{corollary p-adic linearization disc}.  

\begin{theorem}\label{theorem-minimalityQp}
Let $f\in \mathcal{F}_{\lambda,a}\cap \mathbb{Q}_p[[x]]$ for some
prime $p$. Then, $f$ is minimal on each sphere
$S\subset\mathbb{Q}_p$ of radius $r<1/a$ about the origin, if and
only if $\lambda $ is maximal and $p$ is odd. Moreover, if $\lambda $ is
maximal and $p$ is odd, then $g(S)=f(S)=S$.
\end{theorem}

\begin{proof}
Suppose that $f$ is minimal on a rational sphere $S\subset \Delta
_f(0,\mathbb{Q}_p)$ about the origin. In view of Theorem
\ref{theorem-minimalitypreserved}, $T_{\lambda}$ is minimal on
$g(S)$.  Hence, by the conjugacy relation, $f^{\circ
n}(g^{-1}(x))=g^{-1}(\lambda ^n(x))$, the image $g^{-1}(S)$ is
dense in $S$. By continuity of $g^{-1}$ and compactness of $S$,
$g^{-1}(S)=S$. Accordingly, $g(S)=S$. It follows that $T_{\lambda
}$ is minimal on $S$. Consequently, $\lambda $ is maximal.

On the other hand, suppose that $\lambda $ is maximal.
Then, $T_{\lambda}$ is minimal on each sphere $S_r(0)\subset
\mathbb{Q}_p$. In view of Corollary \ref{corollary p-adic linearization disc}, 
the semi-conjugacy relation $g(f^{\circ n}(x))=\lambda ^n
g(x)$, holds for all $x\in \Delta_f(0,\mathbb{Q}_p)\supseteq
D_{1/a}(0)$. By similar arguments as above, we conclude that
$g(S_r(0))=S_r(0)$ if $r<1/a$. It follows that $T_{\lambda}$ is
minimal on $g(S_r(0))$ for $r<1/a$. Consequently, $f\in
\mathcal{F}_{\lambda,a}\cap \mathbb{Q}_p[[x]]$ is minimal on
$S_r(0)$ for each $r<1/a$.

\end{proof}


As shown in Section \ref{section non-Archimedean power series},
$f:\Delta_f\to \Delta_f$ is not only one-to-one, but also
surjective, in the algebraic closure $\mathbb{C}_p$. By Theorem
\ref{theorem-minimalityQp}, $f$ may also be surjective on $\Delta
_f(\mathbb{Q}_p)=\Delta_f \cap \mathbb{Q}_p$.

\begin{corollary}\label{corollary surjective in Qp}
Let $f\in \mathcal{F}_{\lambda,a}\cap \mathbb{Q}_p[[x]]$ for some
odd prime $p$. Suppose that $\lambda $ is maximal, then $f:\Delta_f(0,\mathbb{Q}_p)
\to \Delta_f(0,\mathbb{Q}_p)$ is bijective.
\end{corollary}

\subsection{Unique ergodicity}
Let $K$ be a complete non-Archimedean field. In the following $X$
will be a compact subset of $K$. For example, if $K=\mathbb{C}_p$,
$X$ could be any disc or any sphere in a finite extension of
$\mathbb{Q}_p$.

A continuous map $T:X\to X$ is \emph{uniquely ergodic} if there
exists only one probability measure $\mu$, defined on the Borel
$\sigma$-algebra $\mathcal{B}(X)$ of $X$, such that $T$ is
measure-preserving (i.e.\@ $\mu (A)=\mu (T^{-1}(A))$ for all $A\in
\mathcal{B}(X)$) with respect to $\mu$. As $\mu $ is unique it
follows that $T$ must be ergodic with respect to $\mu$. Recall
that $T$ is ergodic if for any $A\in \mathcal{B}$, whenever
$T^{-1}(A)=A$, then $\mu(A)\mu(A^{c})=0$. (Here $A^{c}$ denotes
the complement of $A$.) 

As shown by Oxtoby 
\cite{Oxtoby:1952}, a bijective isometry of a
compact metric space is uniquely ergodic, hence ergodic, if and
only if it is minimal. See Bryk and Silva \cite{Bryk/Silva:2003}
for a shorter proof in the $p$-adic setting.


In view of Lemma \ref{lemma linearization disc isometry} and Theorem
\ref{theorem-minimalitypreserved}, $f(x)\in\lambda x+
O(x^2)\in K[[x]]$ is certainly bijective and isometric
on compact invariant subsets of a linearization disc $\Delta_f$ about the
origin in $K$. Consequently, transitivity, minimality, ergodicity
and unique ergodicity are all equivalent and preserved under
analytical conjugation on compact subsets of $\Delta_f$.

\begin{theorem}\label{theorem minimlity ergodicity subset}
Let $K$ be a complete non-Archimedean field. Let the power series  $f(x)=\lambda x
+O(x^2)\in K[[x]]$ be analytically conjugate to
$T_{\lambda}$, on a linearization disc $\Delta_f$ about the origin, via a
conjugacy function $g$, such that $g(0)=0$ and $g'(0)=1$. Suppose that
the subset $X\subset \Delta_f$ is non-empty, compact and invariant
under $f$. The following statements are equivalent:
\begin{enumerate}[1.]
      \item $T_{\lambda} : g(X)\to g(X) $ is minimal.
      \item $f: X\to X$ is minimal.
      \item $f: X\to X$ is uniquely ergodic.
      \item $f$ is ergodic for any $f$-invariant measure $\mu$ on $\mathcal{B}(X)$ that is
      positive on non-empty open sets.
    \end{enumerate}

\end{theorem}

Now, we return to the $p$-adic case. Note that a rational sphere
$S\subset\mathbb{C}_p$ is not compact since $\mathbb{C}_p$ is not
locally compact. (Recall that a sphere $S\subset K$ is rational if
and only if it is non-empty, i.e.\@ the radius is a number in the
value group $|K^*|$.)
However, in $\mathbb{Q}_p$ (or any finite extension of
$\mathbb{Q}_p$), every rational sphere is compact. By Theorem
\ref{theorem-minimalityQp}  we have the following result.

\begin{theorem}\label{theorem generator minimality ergodicity}
Let $f\in \mathcal{F}_{\lambda,a}\cap \mathbb{Q}_p[[x]]$ for some
odd prime $p$. Let $S\subset\mathbb{Q}_p$ be a rational sphere of
radius $r<1/a$ about the origin. Then, the following statements
are equivalent:
\begin{enumerate}[1.]
      \item $\lambda $ is maximal.
      \item $f:S\to S$ is minimal.
      \item $f: S\to S$ is uniquely ergodic.
      \item $f$ is ergodic for any $f$-invariant measure $\mu$ on $\mathcal{B}(S)$ that is
      positive on non-empty open sets.
    \end{enumerate}

\end{theorem}
As proven in
\cite{Bryk/Silva:2003,Coelho/Parry:2001,Oselies/Zieschang:1975},
in this case the unique invariant measure $\mu$ is the normalized
Haar measure $\mu$ for which the measure of a disc is equal to the
radius of the disc.

Note that the estimate of the radius $1/a$ is maximal in the sense
that there exist examples of such $f$ which diverges or have a
zero on the sphere $S_{1/a}(0)$, see Lemma \ref{lemma  f
one-to-one} and Corollary \ref{corollary p-adic
linearization disc}.

\section*{Acknowledgements}

I would like to thank
Prof. Andrei Yu. Khrennikov for introduction to the
theory of $p$-adic dynamical systems,the formulation of the
problem on conjugate maps for $p$-adic dynamical systems, and
 for introducing me to the the work of Bryk and Silva \cite{Bryk/Silva:2003}  that helped
me solve the problem on ergodicity. I also thank C.E.
Silva for sending me preprints on their work on $p$-adic
ergodicity.

\addcontentsline{toc}{section}{References}

\bibliographystyle{plain}


\end{document}